%% file: main.tex
\documentclass[11pt,reqno]{amsart}
\usepackage{rotating} % for the rotated table (only works if loaded before the preprint style)
\usepackage{preprintStyle}
\usepackage{makecell}
\usepackage{cancel}

\usepackage{bm} % bold math symbols

\input{variables}
\input{glossaries}

\input{figures/ph_figure_config.tex}

\pgfplotsset{
    table/search path={data}
}

\usepackage{tcolorbox} % for data/code information box
\newtcolorbox{data_code_information_box}{
    colback = black!1!white,
    colframe = black!10!white,
    left=2pt,
    right=2pt,
    top=2pt,
    bottom=2pt,
}

\usepackage[backend=biber,%
			bibencoding=utf8,%
			style=trad-plain,
			giveninits=true,%
			isbn=false,%
			sortcites=true,%
			natbib=true]{biblatex}

\newbibmacro{string+doiurl}[1]{%
  \iffieldundef{doi}{%
    \iffieldundef{url}{%
          #1%
        }{%
      \href{\thefield{url}}{#1}%
    }%
  }{%
    \href{https://doi.org/\thefield{doi}}{#1}%
  }%
}
\DeclareFieldFormat[book,thesis,incollection,inbook]{title}{\usebibmacro{string+doiurl}{\mkbibemph{#1}}}
\DeclareFieldFormat[article,inproceedings]{title}{\usebibmacro{string+doiurl}{#1}} % \mkbibquote{#1}
   
\DeclareNameFormat{familyonly}{%
  \namepartfamily%
  \usebibmacro{name:andothers}}

\DeclareCiteCommand{\tableauthor}
  {\defcounter{maxnames}{1}}
  {\printnames[familyonly]{author}}
  {\multicitedelim}
  {}
  
\newcommand{\citewithauthor}[1]{\tableauthor{#1}~\cite{#1}}

\title[PG operator inference with application to pH systems]{Petrov--Galerkin operator inference with application to stability-encouraging identification}
\author[J.~Rettberg \and J.~Nicodemus \and H.~Sharma \and B.~Kramer \and J.~Fehr \and  B.~Unger]
{Johannes Rettberg${}^\star$ \and Jonas Nicodemus${}^{\mathsection}$ \and Harsh Sharma${}^\ddagger$ \and Boris Kramer${}^\dagger$ \and Jörg Fehr${}^\star$ \and Benjamin Unger${}^{\mathsection}$}

\address{${}^{\star}$ Institute of Engineering and Computational Mechanics, University of Stuttgart, Pfaffenwaldring 9, 70569 Stuttgart, Germany}
\email{\{johannes.rettberg,joerg.fehr\}@itm.uni-stuttgart.de}

\address{${}^{\mathsection}$ Institute for Applied and Numerical Mathematics, Karlsruhe Institute of Technology, 76131
Karlsruhe, Germany}
\email{\{jonas.nicodemus,benjamin.unger\}@kit.edu}

\address{${}^\dagger$ Department of Mechanical and Aerospace Engineering, University of California San Diego, La Jolla, 92093-0411, USA}
\email{bmkramer@ucsd.edu}

\address{${}^\ddagger$ Department of Mechanical Engineering, University of Wisconsin-Madison, 1513 University Ave
Madison, WI, USA}
\email{hsharma33@wisc.edu}

\date{\today}
\keywords{LTI systems, Port-Hamiltonian systems, structure-preserving model reduction, operator inference, data-driven modeling}

\begin{document}

\begin{abstract}
Data-driven model order reduction methods such as operator inference enable the efficient construction of reduced-order models directly from high-dimensional time-domain data.
Standard operator inference typically seeks a Galerkin-type reduced model in a prescribed low-dimensional subspace by identifying its reduced operators from projected snapshot data.
The resulting inference problem is formulated as a least-squares problem admitting an efficient closed-form solution.
However, it is well known from intrusive model order reduction for linear time-invariant systems that Petrov--Galerkin projections can additionally preserve important system properties such as stability and passivity.
To overcome the limitations of standard operator inference, we extend the framework for linear time-invariant systems to incorporate Petrov--Galerkin projections and provide explicit error expressions and bounds between the intrusive and nonintrusive reduced operators, thus generalizing results from the literature. 
We demonstrate the proposed approach in the context of dissipative and port-Hamiltonian systems. Furthermore, we introduce a novel convex optimization formulation that explicitly enforces the port-Hamiltonian structure on the inferred operators.
The effectiveness of the proposed methods is demonstrated on several well-established benchmark problems, including the CD player, an atmospheric model, a mass-spring-damper system, and a poroelasticity system.
\end{abstract}

\maketitle
% {\footnotesize \textsc{Keywords:} LTI systems, passivity, passivation, KYP lemma, $\calH_2$-optimization, gradient-based optimization}

% {\footnotesize \textsc{AMS subject classification:} 37J06, 37M99, 65P10, 93A30, 93B30, 93C05}
%
%30E05 = Moment problems and interpolation problems in the complex plane
%37J06 = General theory of finite-dimensional Hamiltonian and Lagrangian systems, Hamiltonian and Lagrangian structures, symmetries, invariants
%37M99 = Approximation methods and numerical treatment of dynamical systems
%65P10 = Hamiltonian systems including symplectic integrators (numerical problems in dynamics systems)
%93A15 = Large-scale systems
%93A30 = Mathematical modeling (models of systems, model-matching, etc.)
%93B15 = Realizations from input-output data
%93B30 = System identification
%93B99 = Systems theory
%93C05 = Linear systems (Control systems)
%-----------------------------------------------------------------------------%

\glsresetall % reset abbreviations from abstract
% =============================================================================
\section{Introduction}
\label{sec:intro}
% =============================================================================

\Gls{mor} is an essential tool for the efficient simulation and control of large-scale dynamical systems arising in engineering and scientific computing. Classical projection-based \gls{mor} methods require intrusive access to the governing system operators and construct reduced models by projecting the high-dimensional dynamics onto a low-dimensional trial space. In many practical scenarios, such as black-box simulators or legacy codes with heuristic look-up tables, however, the underlying operators are not directly accessible, and only simulation or measurement data are available. This situation has motivated the development of nonintrusive \gls{mor} techniques, among which \gls{opinf} \cite{PeherstorferWillcox16} has gained considerable attention. \gls{opinf} identifies reduced-order system operators directly from projected time-domain data by solving a linear least-squares problem and has been successfully applied in a variety of applications. We refer to \cite{KramerEtAl24} for a recent survey.

Most \gls{opinf} formulations rely on a Galerkin projection, where the trial and test spaces coincide. In contrast, Petrov--Galerkin projections play a central role in classical projection-based \gls{mor}, for instance, in interpolatory methods and structure-preserving model reduction.
Moreover, relying on Petrov--Galerkin projection may improve the predictive performance of the resulting surrogate model; cf.~\cite{OttPR22,OttPR23}.
Despite their importance in intrusive \gls{mor}, the connection between Petrov--Galerkin projection techniques from classical \gls{mor} and \gls{opinf} remains largely unexplored. 

Recently, \cite{PVB24} proposed a data-driven Petrov--Galerkin model reduction framework in which the projection operators and reduced dynamics are learned simultaneously through a nonlinear optimization problem. In contrast, the present work extends the classical \gls{opinf} methodology to prescribed Petrov--Galerkin trial and test spaces for linear dynamical systems. Rather than identifying the projection itself, we infer the reduced operators associated with a prescribed oblique projection. The resulting problem retains the computational efficiency and linear least-squares structure of standard \gls{opinf} and therefore admits a unique solution under suitable rank conditions. This contrasts with the joint nonlinear optimization of projection operators and reduced dynamics in \cite{PVB24}, where the optimization landscape may contain local minima. At the same time, our formulation preserves a direct connection to intrusive Petrov--Galerkin \gls{mor}, which enables convergence and error analysis with respect to the corresponding intrusive reduced models. Moreover, in structure-preserving settings such as \gls{ph} systems, the trial and test spaces can be chosen based on physical and energetic considerations rather than learned through optimization, thereby naturally incorporating system-theoretic properties into the data-driven reduction process.

% -----------------------------------------------------------------------------------
\subsection{Contributions}
\label{sec:intro:contributions}
% -----------------------------------------------------------------------------------
Our main contributions can be grouped into two categories. First, we derive a theoretical framework for Petrov--Galerkin \gls{opinf}, with the following main results:
\begin{enumerate}[label=(\roman*)]
	\item \emph{Petrov--Galerkin operator inference}: We introduce a Petrov--Galerkin variant of \gls{opinf} in \Cref{sec:pgopinf} and in \Cref{thm:pgopinf}, we present an operator decomposition result that demonstrates that the nonintrusive reduced operators are given by the intrusive reduced operators augmented by a correction term that accounts for the influence of unresolved state components, thereby relating it to the reprojection framework \cite{Peherstorfer20}.
	\item \emph{Error characterization}: We demonstrate that it is not sufficient that  the norm of the unresolved state components converges to zero for the nonintrusive operators to converge to the intrusive operators, as simultaneously, the least-squares problem becomes more ill-conditioned. We then derive an explicit error bound between the intrusive and nonintrusive operators; see \Cref{thm:OpInfNoConv}. 
\end{enumerate}
Second, to demonstrate the potential of Petrov--Galerkin \gls{opinf}, we rely on structure-preserving \gls{mor} methods for \gls{ph} systems, thereby obtaining the following contributions:
\begin{enumerate}[label=(\roman*)]
	\item \emph{Stability-encouraging operator inference}: Leveraging insights from \gls{ph} systems, we employ in \Cref{sec:method:ph_petrov_galerkin_opinf} a Petrov--Galerkin projection that is known from structure-preserving nonintrusive \gls{mor} for linear \gls{ph} systems. This choice of projection promotes stability properties of the inferred reduced model. Since constructing the corresponding test space requires knowledge of the quadratic Hamiltonian, we propose an energy-matrix inference problem that estimates the Hamiltonian matrix from sampled Hamiltonian data.
	\item \emph{Convex \gls{ph} identification}: We further develop in \Cref{sec:method:ph_convex_optimization} a one-step approach for identifying a low-dimensional linear \gls{ph} system based on state, input, and output data by solving a convex optimization problem.
	Unlike existing methods, we do \emph{not} 
	\begin{itemize}
		\item require knowledge of the underlying Hamiltonian,
		\item need to solve a nonlinear optimization problem, or
		\item rely on a process of fixing one matrix and optimizing the other.
	\end{itemize}
\end{enumerate}
We demonstrate our theoretical results and the potential of Petrov--Galerkin \gls{opinf} over standard \gls{opinf} on several well-studied benchmark problems in \Cref{sec:numerics}.
% -----------------------------------------------------------------------------------
\subsection{Literature Review and State of the Art}
\label{sec:intro:literature}
% -----------------------------------------------------------------------------------
% General data-driven modeling/system identification methods

System identification, i.e., the derivation of models from measurements and observations, is a core methodology in the natural sciences.
The desire to do this in an automated fashion has long been present, and early works can be found in \cite{AastroemEykhoff71,NarendraParthasarathy90,Ljung12}. As the current body of literature, including many recent advanced using neural network architectures, is quite significant, we focus our literature review here solely to the core concept required for the understanding of our work, i.e., \gls{opinf} and closely related work, structure-preserving intrusive \gls{mor} for \gls{ph} systems, and their nonintrusive counterparts.

%%%%%%%%%%%%%%%%%%%%%%%%%%%%%%%%%%%%%%%%
% nonintrusive MOR, including DMD and OpInf
\subsubsection{Nonintrusive Model Order Reduction}
% \label{sec:intro:literature:non_intrusive_mor}
%%%%%%%%%%%%%%%%%%%%%%%%%%%%%%%%%%%%%%%%
For high-dimensional systems, surrogate models based on \gls{mor} can be employed to enable efficient simulation and system identification.
A distinction is made between intrusive \gls{mor}, where system operators are accessible, and nonintrusive \gls{mor}, where only data is available.
The latter commonly occurs when data originates from commercial closed-source software with restricted access to the underlying operators.
Probably the best-known approach here is the autoencoder, which maps the data nonlinearly into a latent space of smaller dimensions and then decodes it back into physical space for interpretation.
Various directions have been investigated in \cite{FrescaEtAl21,GruberEtAl22,PichiEtAl24,KneiflEtAl23,ContiEtAl24,ChampionEtAl19}. 
Another well-known method for identifying reduced system operators is \gls{dmd} \cite{Schmid10}.
The goal is to find the spectral decomposition of a linear operator that best continues the system state from one snapshot to the next \cite{BruntonKutz22} with extensions in \cite{ProctorEtAl16,WilliamsEtAl15}.
The \gls{opinf} framework, first proposed in \cite{PeherstorferWillcox16}, extends the \gls{dmd} method by leveraging state-derivative data, accommodating parameter-dependent operators, and enabling the direct identification of output operators.
Unlike the original \gls{dmd} approach, \gls{opinf} can represent models containing nonlinear polynomial terms \cite{KramerEtAl24}, which, in combination with a lifting transformation, facilitates the treatment of complex nonlinear dynamics \cite{QianEtAl20}.

%%%%%%%%%%%%%%%%%%%%%%%%%%%%%%%%%%%%%%%%
% port-Hamiltonian systems
\subsubsection{Structure-preserving Model Order Reduction of Port-Hamiltonian Systems}
%%%%%%%%%%%%%%%%%%%%%%%%%%%%%%%%%%%%%%%%
\Acrlong{ph} systems are attracting increasing attention and a growing number of scientists are conducting research in this field.
Overview papers exist for \glspl{ode} \cite{SchaftJeltsema14}, \glspl{dae} \cite{MehrmannUnger23}, and \glspl{pde} \cite{RashadEtAl20}.
The focus is on the mathematical structure and its utility for areas such as control \cite{GarciaCansecoEtAl05,PhilippEtAl24}, interconnectivity \cite{CerveraEtAl07}, and discretization in space \cite{CardosoRibeiroEtAl20} and in time \cite{KotyczkaLefevre19}.
% Structure-preserving intrusive MOR
Structure-preserving intrusive \gls{mor} for \gls{ph} systems has been adapted for many well-known \gls{mor} methods.
There are structure-preserving variants of tangential interpolation \cite{GugercinEtAl12}, moment-matching \cite{IonescuAstolfi13,PolyugavanderSchaft10}, or balanced truncation \cite{BreitenEtAl22}.
Reduction approaches based on the differential geometric Dirac structure have been developed in \cite{HauschildEtAl19,PolyugavanderSchaft12}.
Optimization-based methods are described in \cite{MoserLohmann20,SchwerdtnerVoigt23}.
In this work however, we focus on nonintrusive methods.

%%%%%%%%%%%%%%%%%%%%%%%%%%%%%%%%%%%%%%%%
% port-Hamiltonian identification
\subsubsection{System Identification of Port-Hamiltonian Systems}
\label{sec:intro:literature:ph_system_identification}
%%%%%%%%%%%%%%%%%%%%%%%%%%%%%%%%%%%%%%%%
The identification and data-based modeling of \gls{ph} systems is both challenging and beneficial. Due to the close relation between the \gls{ph} formulation and physics, \gls{ph} system identification naturally incorporates physical knowledge into the identification process, increasing the likelihood of obtaining robust models that comply with physical laws and generalize well. Existing approaches can be roughly distinguished by the considered domain, required data, and type of \gls{ph} system. An overview of selected \gls{ph} identification methods in comparison with the approaches developed in this work is given in \cref{tab:ph_identification_methods}.

For \textit{linear} \gls{ph} systems, \cite{BennerEtAl20} proposed a Loewner-based method for frequency-domain data that directly yields a transfer function satisfying the \gls{ph} structure, but requires accurate estimation of the spectral zeros. In contrast, \cite{Schwerdtner21} formulates a nonconvex optimization problem based on the difference between transfer functions. In \cite{CherifiEtAl22}, frequency-domain data is first estimated from time-domain measurements using tailored inputs before applying the Loewner framework of \cite{BennerEtAl20}.

A multi-step time-domain approach is proposed in \cite{CherifiEtAl19}. First, conventional \gls{opinf} is used to identify a standard state-space model, which is subsequently transformed into a \gls{ph} representation. Since this transformation requires the Kalman-Yakubovich-Popov inequality to admit a solution, the authors additionally suggest approximating the nearest positive-real system using \cite{GillisSharma18}. This approach, however, is sensitive to the initial guess and convergence is not guaranteed.
The authors of \cite{GuentherEtAl24} fit \gls{ph} matrices from basis elements using gradient-based sensitivity optimization, while \cite{GuentherEtAl24a} employs optimal-control techniques to solve a constrained nonconvex calibration problem. In \cite{GoyPB25,GoyalEtAl23a}, linear \gls{ph} systems without feedthrough are identified by avoiding explicit structural constraints through a reparametrization first introduced in \cite{Mlinaric20}. This yields an unconstrained problem at the cost of increased nonconvexity. In a similar spirit, \cite{EPB26} identifies stable systems using the same reparametrization while additionally optimizing the oblique projection matrices.
More recently, the \gls{ph} \gls{dmd} framework was introduced in \cite{MorandinEtAl23}. It modifies the \gls{dmd} problem to explicitly incorporate system structure and iteratively optimizes the structure and dissipation matrices, but requires access to the Hessian of the Hamiltonian.

The identification of \textit{nonlinear} \gls{ph} systems from time-domain data has been addressed in fewer works and typically under restrictive assumptions. For instance, \cite{GJKW25} assumes the nonlinear Hamiltonian to be known, while \cite{DesaiEtAl21} uses a neural network to identify a \gls{ph} structure under assumptions on the form of the dissipation and forcing terms. In \cite{RettbergEtAl25}, an autoencoder nonlinearly maps the states to a latent space in which a linear \gls{ph} system is subsequently identified.

\begin{table}[htb]
    %\centering
    \caption{A selection of related \gls{ph} identification methods in the literature in comparison with the approaches described in this work.
	}
    \footnotesize
    \label{tab:ph_identification_methods}
    \begin{tabular}{lllllp{5cm}}
        \toprule
        reference & domain & data & system class & key ingredient \\ \midrule
        \citewithauthor{BennerEtAl20} & frequency & $\transferFunction(\laplaceVariable)$ samples%, spectral zeros 
        & linear & pH Loewner\\ 
        \citewithauthor{Schwerdtner21} & frequency & $\transferFunction(\laplaceVariable)$ samples & linear, DAE & nonlinear optimization \\
        \citewithauthor{CherifiEtAl19} & time  & $\phOutputs$, $\phInputs$ & linear, DAE & mod.~Loewner + nearest PR \cite{GillisSharma18} \\ 
        \citewithauthor{CherifiEtAl22} & time & $\phOutputs$, $\phInputs$ & linear & mod.~pH Loewner \\ 
        \citewithauthor{GoyPB25} & time  & $\states$, $\phOutputs$, $\phInputs$ & linear & nonlinear optimization\\ 
        \citewithauthor{MorandinEtAl23} & time & $\states$, $\phOutputs$, $\phInputs$, $\phQ$ & linear & projected gradient optimization \\
        \citewithauthor{GuentherEtAl24} & time & $\phOutputs$, $\phInputs$ & linear & sensitivity analysis \\
        \citewithauthor{GuentherEtAl24a} & time & $\phOutputs$, $\phInputs$ & linear & constrained optimization % nonconvex 
        \\
		\citewithauthor{RettbergEtAl25} & time & $\states$, $\phOutputs$, $\phInputs$ & (non)linear & autoencoder + NN \\
		\citewithauthor{DesaiEtAl21} & time & $\states$ % - gen.~position and momentum 
        & nonlinear & dissipation and forcing structure \\
		\citewithauthor{GJKW25} & time & $\states$, $\phOutputs$, $\phInputs$, $\nabla\hamiltonian$ & nonlinear & operator inference\\ 
		\midrule
		\pgopinf
		 & time & $\states$, $\phOutputs$, $\phInputs$, $\phQ$ & linear &  see \cref{sec:method:ph_petrov_galerkin_opinf}\\
		\pgopinfham
		& time & $\states$, $\phOutputs$, $\phInputs$, $\hamiltonian$ & linear &  
		see \cref{sec:method:hamiltonian_based_phQ_identification}
		\\
		\cvxph
		& time & $\states$, $\phOutputs$, $\phInputs$ & linear & see \cref{sec:method:ph_convex_optimization} \\
        \bottomrule        
    \end{tabular}
\end{table}

% -----------------------------------------------------------------------------------
\subsection{Organization of the Manuscript}
% -----------------------------------------------------------------------------------
Our paper is organized as follows. In \cref{sec:prelims}, we present the fundamental concepts required for this work. In particular, we review the basics of the \gls{ph} framework, including its intrinsic stability and passivity properties, and discuss both intrusive and non-intrusive \gls{mor} approaches.
Our main contributions are developed in \cref{sec:methodology}. We begin by introducing the Petrov--Galerkin projected \gls{opinf} method in \cref{sec:pgopinf}. This projection is then applied in the context of \gls{ph} systems in \cref{sec:method:linear_ph_identification}, where we present different approaches for deriving the required Hessian of the Hamiltonian. In addition, we introduce a convex optimization formulation that does not impose additional requirements.
In \cref{sec:numerics}, we investigate the performance and mathematical validity of the proposed method through numerical examples.
Finally, \cref{sec:conclusions} summarizes the main contributions and findings of this work.

% -----------------------------------------------------------------------------------
\subsection{Notation}
% -----------------------------------------------------------------------------------
We use the symbols $\N$, $\R$, $\C$, $\Cplus$, $\R^n$, $\R^{n\times m}$, and $\eye{n}$ to denote, respectively, the positive integers, the real numbers, the complex numbers, the complex numbers with positive real part, the set of column vectors with $n\in\N$ real entries, the set of real $n\times m$ matrices, and the $n$-dimensional identity matrix.
Furthermore, $\stateSpaceSystemMatrix \succ 0$ and $\stateSpaceSystemMatrix \succeq 0$ indicate that the matrix $\stateSpaceSystemMatrix\in\R^{n\times n}$ is, respectively, symmetric positive definite and symmetric positive semidefinite.
We denote with $\Spd{n}$, and $\Spsd{n}$ the sets of real $n \times n$ symmetric positive definite and symmetric positive semidefinite  matrices, respectively.
We use the symbols $\stateSpaceSystemMatrix\tran$, and $\frobNorm{\stateSpaceSystemMatrix}$ for the transpose and the Frobenius norm of a matrix, respectively.

% =============================================================================
\section{Preliminaries}
\label{sec:prelims}
% =============================================================================

In this work, we consider dynamical systems that can be modeled via \gls{lti} systems of the form
\begin{equation}
	\label{eq:full_order_state_space_system}
	\begin{aligned}
		\Dstates(\timee) 
			&= \stateSpaceSystemMatrix \states(\timee) + \stateSpaceInputMatrix\stateSpaceInputs(\timee), 
			\quad \states(0) = \initialConditionStates, \\
		\stateSpaceOutputs(\timee) 
			&= \stateSpaceOutputMatrix\states(\timee) + \stateSpaceFeedthroughMatrix\stateSpaceInputs(\timee),
	\end{aligned}
\end{equation} 
with a state $\states\colon\timeInterval\to\R^{\dimStates}$, the initial condition $\initialConditionStates\in\R^{\dimStates}$, the input $\stateSpaceInputs\colon\timeInterval\to\R^{\dimInputs}$, and the outputs $\phOutputs\colon\timeInterval\to\R^{\dimOutputs}$ for a time interval $\timeInterval\vcentcolon=\brackets{0,\timeeEnd}$ with $\timeeEnd>0$ and $\timee\in\timeInterval$.
While the state dimension $\dimStates \in \N$ can be large, the numbers of inputs $\dimInputs \in \N$ and outputs $\dimOutputs \in \N$ are typically, comparatively small.
Here, $\stateSpaceSystemMatrix\in\R^{\dimStates\times\dimStates}$, $\stateSpaceInputMatrix\in\R^{\dimStates\times\dimInputs}$, $\stateSpaceOutputMatrix\in\R^{\dimOutputs\times\dimStates}$, $\stateSpaceFeedthroughMatrix\in\R^{\dimOutputs\times\dimInputs}$ denote the system, input, output, and feedthrough matrix, respectively.
We refer to the \gls{lti} system~\cref{eq:full_order_state_space_system} as the \gls{fom} or high-fidelity model.
It can be represented in the frequency domain via its transfer function $\transferFunction:\C\setminus\spectrum(\stateSpaceSystemMatrix)\to\C^{\dimOutputs\times\dimInputs}$ as
\begin{align}
	\label{eq:lti_transfer_function}
	\transferFunction(\laplaceVariable) = \stateSpaceOutputMatrix(\laplaceVariable \eye{\dimStates} - \stateSpaceSystemMatrix)\inv\stateSpaceInputMatrix + \stateSpaceFeedthroughMatrix,
\end{align}
where $\laplaceVariable\in\C$ is the complex Laplace variable and $\spectrum(\stateSpaceSystemMatrix)$ denotes the spectrum of the system matrix $\stateSpaceSystemMatrix$.

% -----------------------------------------------------------------------------------
\subsection{Dissipative-Hamiltonian and Port-Hamiltonian Systems}
% -----------------------------------------------------------------------------------
The system \cref{eq:full_order_state_space_system} is known to be \emph{stable} if all eigenvalues $\eigenvalue\in\spectrum(\stateSpaceSystemMatrix)$ lie in the left half of the complex plane, i.e., $\real{\eigenvalue}\leq0$ and all eigenvalues on the imaginary axis have equal algebraic and geometric multiplicities. 
Furthermore, the system is \emph{asymptotically stable} if all eigenvalues lie strictly in the left half of the complex plane, i.e., $\real{\eigenvalue}<0$; see for instance \cite{HinrichsenPritchard05}.

\begin{definition}[Dissipative Hamiltonian decomposition]
	Let $\stateSpaceSystemMatrix\in\R^{\dimStates\times\dimStates}$ be the system matrix of an asymptotically stable \gls{lti} system~\cref{eq:full_order_state_space_system}.
	The decomposition of $\stateSpaceSystemMatrix$ into the form
	\begin{equation*}
	\stateSpaceSystemMatrix = (\phJ - \phR)\phQ,
	\end{equation*}
	with a skew-symmetric matrix $\phJ\in\R^{\dimStates\times\dimStates}$, a symmetric positive semidefinite matrix $\phR\in\Spsd{\dimStates}$, and a symmetric positive definite matrix $\phQ\in\Spd{\dimStates}$, is referred to as \emph{dissipative Hamiltonian} decomposition.
\end{definition}
In \cite{GillisSharma17}, it is shown that the system~\cref{eq:full_order_state_space_system} is asymptotically stable if and only if the system matrix admits a dissipative Hamiltonian decomposition.

A system \cref{eq:full_order_state_space_system} is called \emph{passive} if there exists a continuously differentiable storage function $\storageFunction:\R^{\dimStates}\to\R$, the so-called \emph{Hamiltonian}, such that for all $\timee\in\timeInterval$ the dissipation inequality
\begin{align}
	\label{eq:dissipation_inequality}
	\ddt \storageFunction(\states(\timee)) \leq \stateSpaceInputs(\timee)\tran \stateSpaceOutputs(\timee)
\end{align}
for $\dimOutputs=\dimInputs$ holds \cite{Willems72,Willems72a}.

\begin{definition}[\acrlong{ph} system]
An \gls{lti} system of the form
\begin{equation}
\label{eq:linear_ph_system_with_feedthrough}
\begin{aligned}
\Dstates(\timee) &= (\phJ-\phR)\phQ\states(\timee) + (\phInputMatrix - \phFeedthroughInputMatrix) \phInputs(\timee), \quad \states(0) = \initialConditionStates, \\
\phOutputs(\timee) &= (\phInputMatrix + \phFeedthroughInputMatrix)\tran\phQ\states(\timee) + (\phFeedthroughOutputMatrixSym + \phFeedthroughOutputMatrixSkew)\phInputs(\timee),
\end{aligned}
\end{equation}
together with the Hamiltonian
\begin{align}
	\label{eq:hamiltonian_quadratic_form}
	\hamiltonian\colon \R^{\dimStates}\to\R, \quad \states \mapsto \tfrac{1}{2}\states\tran\phQ\states,
\end{align} 
with the matrices $\phJ,\phR,\phQ\in\R^{\dimStates\times\dimStates}$, $\phInputMatrix, \phFeedthroughInputMatrix\in\R^{\dimStates\times\dimInputs}$, $\phFeedthroughOutputMatrixSym,\phFeedthroughOutputMatrixSkew\in\R^{\dimInputs\times\dimInputs}$, is called a \gls{ph} system if the following conditions hold:
\begin{enumerate}[label=(\roman*)]
	\item the Hessian of the Hamiltonian $\phQ$ is symmetric positive definite, 
	\item the structure matrix $ \phStructureMatrix \vcentcolon= \begin{bmatrix}
\phJ & \phInputMatrix \\
-\phInputMatrix\tran & \phFeedthroughOutputMatrixSkew
\end{bmatrix}$ is skew-symmetric, and
	\item the dissipation matrix $\phGeneralDissipationMatrix \vcentcolon= \begin{bmatrix}
\phR & \phFeedthroughInputMatrix \\
\phFeedthroughInputMatrix\tran & \phFeedthroughOutputMatrixSym
\end{bmatrix}$ is symmetric positive semidefinite.
\end{enumerate}
\end{definition}

Using the block matrices $\phStructureMatrix$  and $\phGeneralDissipationMatrix$, we can equivalently write \cref{eq:linear_ph_system_with_feedthrough} as
\begin{align}
    \label{eq:linear_block_ph_system_with_feedthrough}
        \begin{bmatrix}
        \Dstates(\timee) \\
        -\phOutputs(\timee)
    \end{bmatrix} = (\phStructureMatrix-\phGeneralDissipationMatrix)\begin{bmatrix}
        \phQ\states(\timee) \\
        \phInputs(\timee)
    \end{bmatrix}.
\end{align}
From these structural constraints one can show that \gls{ph} systems are passive and stable; cf.~\cite{SchaftJeltsema14}.

\subsection{Intrusive Model Order Reduction}
In the case that the system operators of \cref{eq:full_order_state_space_system} are known, the common system-theoretic approach of surrogate modeling is projection-based \gls{mor}.
Here, the dynamics are approximated in an $\dimReducedStates$-dimensional subspace $\reducedSpace\subseteq\R^{\dimStates}$ defined by a projection basis $\reductionBasis\in\R^{\dimStates\times\dimReducedStates}$ such that $\Span{\reductionBasis}=\reducedSpace$, where the reduced dimension $\dimReducedStates\in\N$ is usually much lower than the original dimension, i.e., $\dimReducedStates\ll\dimStates$.
By inserting the approximation ansatz 
\begin{align}
	\states(\timee)\approx \reductionBasis\reducedStates(\timee),
\end{align}
with the reduced state $\reducedStates\colon\timeInterval\to\R^{\dimReducedStates}$ into \cref{eq:full_order_state_space_system}, and enforcing the residual to be orthogonal to the test space $\reducedSpaceTest\subseteq\R^{\dimStates}$ defined by a second projection matrix $\reductionBasisLeft\in\R^{\dimStates\times\dimReducedStates}$ such that $\Spann\big(\reductionBasisLeft\big)=\reducedSpaceTest$ and $\reductionBasisLeft\tran \reductionBasis$ has full column rank (i.e., is nonsingular), we obtain the \gls{rom} of order $\dimReducedStates$ as
\begin{subequations}\label{eq:reduced_order_state_space_system}
	\begin{equation}\label{eq:reduced_order_state_space_system:sys}
	\begin{aligned}
	 \stateSpaceDescriptorMatrixReduced \DreducedStates(\timee)
        &= \stateSpaceSystemMatrixReduced\reducedStates(\timee)
        + \stateSpaceInputMatrixReduced\stateSpaceInputs(\timee),
		\quad \reducedStates(0) = \parentheses{\reductionBasisLeft\tran\reductionBasis}\inv\reductionBasisLeft\tran \initialConditionStates, \\
        \stateSpaceOutputsReduced(\timee)
        &= \stateSpaceOutputMatrixReduced \reducedStates(\timee)
        + \stateSpaceFeedthroughMatrixReduced\stateSpaceInputs(\timee),
	\end{aligned}
	\end{equation}
	where the reduced matrices are defined by
	\begin{equation}\label{eq:reduced_order_state_space_system:mats}
	\begin{aligned}
    	\stateSpaceDescriptorMatrixReduced &\vcentcolon=\reductionBasisLeft\tran \reductionBasis, &
    	\stateSpaceSystemMatrixReduced &\vcentcolon=\reductionBasisLeft\tran\stateSpaceSystemMatrix\reductionBasis, &
    	\stateSpaceInputMatrixReduced &\vcentcolon=\reductionBasisLeft\tran\stateSpaceInputMatrix, &
    	\stateSpaceOutputMatrixReduced &\vcentcolon=\stateSpaceOutputMatrix\reductionBasis, &
    	\stateSpaceFeedthroughMatrixReduced &\vcentcolon=\stateSpaceFeedthroughMatrix.
    \end{aligned}
	\end{equation}
\end{subequations}
The system~\cref{eq:reduced_order_state_space_system:sys} is still an \gls{lti} system with the descriptor matrix $\stateSpaceDescriptorMatrixReduced$. It can be transformed into the form of~\cref{eq:full_order_state_space_system} by left-multiplying the first equation of \cref{eq:reduced_order_state_space_system:sys} with $\stateSpaceDescriptorMatrixReduced^{-1}$.

This oblique projection with the underlying projector
\begin{align}
	\label{eq:projector_petrov_galerkin}
	\projector = \reductionBasis\parentheses{\reductionBasisLeft\tran\reductionBasis}\inv\reductionBasisLeft\tran\in\R^{\dimStates\times\dimStates}
\end{align}
is known as \emph{Petrov--Galerkin projection}, while an orthogonal projection with $\reductionBasisLeft:=\reductionBasis$ is called a \emph{Galerkin} projection.
An arbitrary (Petrov--)Galerkin projection does not, in general, preserve the stability or passivity of a system.

\begin{proposition}[\cite{WLEK10}]
	Let $\reductionBasisLeft$ be chosen as $\reductionBasisLeft = \phQ\reductionBasis$. Then, the \gls{rom} is given by 
	\begin{subequations}\label{eq:pg:intrusive:rom:ph}
	\begin{equation}\label{eq:pg:intrusive:rom:ph:sys}
	\begin{aligned}
	\phEReduced\DreducedStates(\timee) 
		&= (\phJReduced-\phRReduced)\reducedStates(\timee) + (\phInputMatrixReduced - \phFeedthroughInputMatrixReduced) \phInputs(\timee), 
		\quad \reducedStates(0) = \parentheses{\reductionBasisLeft\tran\reductionBasis}\inv\reductionBasisLeft\tran \initialConditionStates, \\
	\reducedphOutputs(\timee) 
		&= (\phInputMatrixReduced + \phFeedthroughInputMatrixReduced)\tran\reducedStates(\timee) + (\phFeedthroughOutputMatrixSymReduced + \phFeedthroughOutputMatrixSkewReduced)\phInputs(\timee),
	\end{aligned}
	\end{equation}
	where the reduced  matrices are defined by
	\begin{equation}\label{eq:pg:intrusive:rom:ph:mats}
		\begin{aligned}
		\phEReduced &\vcentcolon= \reductionBasis\tran\phQ\tran\reductionBasis,
		&\phQReduced &\vcentcolon= \eye{\dimReducedStates},
		&\phJReduced &\vcentcolon= \reductionBasis\tran\phQ\tran\phJ\phQ\reductionBasis, 
		&\phRReduced &\vcentcolon= \reductionBasis\tran\phQ\tran\phR\phQ\reductionBasis, \\
		\phInputMatrixReduced &\vcentcolon= \reductionBasis\tran\phQ\tran\phInputMatrix,
		&\phFeedthroughInputMatrixReduced &\vcentcolon= \reductionBasis\tran\phQ\tran\phFeedthroughInputMatrix,
		&\phFeedthroughOutputMatrixSymReduced &\vcentcolon= \phFeedthroughOutputMatrixSym, 
		&\phFeedthroughOutputMatrixSkewReduced &\vcentcolon= \phFeedthroughOutputMatrixSkew.
	\end{aligned} 
	\end{equation}
	\end{subequations}
	Hence, $\phEReduced$ is symmetric positive definite, $\phJReduced$ is skew-symmetric, and $\phRReduced$ is symmetric positive semidefinite.
\end{proposition}

It is well-known, see for instance~\cite{BMXZ18,MehrmannUnger23}, that the \gls{rom}~\cref{eq:pg:intrusive:rom:ph} is a \gls{ph} system with a descriptor matrix $\phEReduced$. However, it can be transformed into the form of $\cref{eq:linear_ph_system_with_feedthrough}$ via the state transformation with $\reducedStates\mapsto\phEReduced\inv\reducedStates$, which yields $\phQReduced=\phEReduced^{-1}$; cf.~\cite[Rem.~20]{BMXZ18}.

\subsection{Operator Inference}
In contrast to intrusive \gls{mor} approaches, the nonintrusive setting assumes that the system operators in \cref{eq:full_order_state_space_system} are not accessible and that only full state data are available.
Within this context, the \gls{opinf} framework seeks to infer low-dimensional operators that best fit data generated by a dynamical system \cite{KramerEtAl24}.
The \gls{opinf} workflow is typically subdivided into three distinct phases.

\textbf{Phase 1:} During the data generation, $\numDataPoints\in\N$ snapshots of the states, state derivatives, inputs and outputs are calculated via a time discretization of \cref{eq:full_order_state_space_system} and stacked into matrices 
\begin{equation}
	\label{eq:data_matrices}
	\begin{aligned}		
	\DstatesData &= \begin{bmatrix}
		\Dstates_{1} & \dots & \Dstates_{\numDataPoints}
	\end{bmatrix} \in\R^{\dimStates\times\numDataPoints},\quad
	&&\statesData&=\begin{bmatrix}
		\states_{1} & \dots & \states_{\numDataPoints}
	\end{bmatrix} \in\R^{\dimStates\times\numDataPoints},\\
	\inputData &= \begin{bmatrix}
		\stateSpaceInputs_{1} & \dots & \stateSpaceInputs_{\numDataPoints}
	\end{bmatrix} \in\R^{\dimInputs\times\numDataPoints}, \quad
	&&\outputData &= \begin{bmatrix}
		\stateSpaceOutputs_{1} & \dots & \stateSpaceOutputs_{\numDataPoints}
	\end{bmatrix} \in\R^{\dimOutputs\times\numDataPoints},
	\end{aligned}
\end{equation}
where the discretized data, i.e., $\states_{\iteratorTimeDiscretization}\approx\states(\timee_{\iteratorTimeDiscretization})$ for $\iteratorTimeDiscretization=1,\dots,\numDataPoints$ with $\timee_{\iteratorTimeDiscretization}\in\timeInterval$, depends on the underlying discretization scheme. The quantities $\Dstates_{\iteratorTimeDiscretization}$, $\stateSpaceInputs_{\iteratorTimeDiscretization}$, and  $\stateSpaceOutputs_{\iteratorTimeDiscretization}$ are defined analogously.
If no data for the time derivatives $\Dstates_{\iteratorTimeDiscretization}$ is available, then a classical finite difference method may be used to approximate $\Dstates_{\iteratorTimeDiscretization}$.

\textbf{Phase 2}: A basis $\reductionBasis\in\R^{\dimStates\times\dimReducedStates}$ for the subspace $\reducedSpace$ is constructed.
Typically, the \gls{pod} method is used on the state data $\statesData$, which finds an orthonormal basis, i.e., $\reductionBasis\tran\reductionBasis=\eye{\dimReducedStates}$, that minimizes the Euclidean distance between the data points and their projection \cite{Volkwein13}.
With this basis, the data is projected as
\begin{align}
	\label{eq:projected_data_matrices_galerkin}
	\reducedStatesData = \reductionBasis\tran\statesData, \qquad\text{and}\qquad
	\reducedDstatesData = \reductionBasis\tran\DstatesData.
\end{align}

\textbf{Phase 3:} In the identification phase, the reduced operators are inferred by solving the minimization problem
\begin{equation}
	\label{eq:opinf_optimization_problem}
	\min_{\stateSpaceSystemMatrixReducedIdentified, \stateSpaceInputMatrixReducedIdentified, \stateSpaceOutputMatrixReducedIdentified, \stateSpaceFeedthroughMatrixReducedIdentified} \quad
	\Bigg\lVert % can't use autosize because of underbrace
		\underbrace{
	\begin{bmatrix}
		\reducedDstatesData \\ \outputData
	\end{bmatrix}}_{=\vcentcolon\lhsDstatesOutputDataReduced}
	-
	\underbrace{
	\begin{bmatrix}
		\stateSpaceSystemMatrixReducedIdentified & \stateSpaceInputMatrixReducedIdentified \\
		\stateSpaceOutputMatrixReducedIdentified & \stateSpaceFeedthroughMatrixReducedIdentified
	\end{bmatrix}}_{=\vcentcolon\blockMatrixStateSpaceABCDReducedIdentified}
	\underbrace{
	\begin{bmatrix}
		\reducedStatesData \\ \inputData
	\end{bmatrix}}_{=\vcentcolon\rhsStatesInputDataReduced}
	\Bigg\rVert_{\indexFrobNorm}^{2},
\end{equation}
or a suitably regularized version thereof, where $\stateSpaceSystemMatrixReducedIdentified\in\R^{\dimReducedStates\times\dimReducedStates}$, $\stateSpaceInputMatrixReducedIdentified\in\R^{\dimReducedStates\times\dimInputs}$, $\stateSpaceOutputMatrixReducedIdentified\in\R^{\dimOutputs\times\dimReducedStates}$, $\stateSpaceFeedthroughMatrixReducedIdentified\in\R^{\dimOutputs\times\dimInputs}$ are the identified counterparts of the reduced system operators from \cref{eq:reduced_order_state_space_system}.
By defining the block matrices $\blockMatrixStateSpaceABCDReducedIdentified\in\R^{(\dimReducedStates+\dimInputs) \times (\dimReducedStates+\dimInputs)}$, $\lhsDstatesOutputDataReduced\in\R^{(\dimReducedStates+\dimInputs) \times \numDataPoints}$, and $\rhsStatesInputDataReduced\in\R^{(\dimReducedStates+\dimInputs) \times \numDataPoints}$, the minimum-norm solution of~\cref{eq:opinf_optimization_problem}, is given by
\begin{align}
	\begin{bmatrix}
		\stateSpaceSystemMatrixReducedIdentifiedMinimizer & \stateSpaceInputMatrixReducedIdentifiedMinimizer \\
		\stateSpaceOutputMatrixReducedIdentifiedMinimizer & \stateSpaceFeedthroughMatrixReducedIdentifiedMinimizer
	\end{bmatrix}=\blockMatrixStateSpaceABCDReducedIdentifiedMinimizer = \lhsDstatesOutputDataReduced\rhsStatesInputDataReduced\pinv,
\end{align}
where $\rhsStatesInputDataReduced\pinv$ denotes the Moore-Penrose pseudoinverse of $\rhsStatesInputDataReduced$.
We emphasize that this solution is unique if and only if the matrix $\rhsStatesInputDataReduced$ has full row rank.

% =============================================================================
\section{Methodology}
\label{sec:methodology}
The proposed Petrov--Galerkin \gls{opinf} methodology leverages the structure-preserving projection~\cref{eq:pg:intrusive:rom:ph} from the intrusive \gls{mor} of \gls{ph} systems in the nonintrusive \gls{opinf} framework.
We begin by presenting the extension of \gls{opinf} to Petrov--Galerkin projected \glspl{rom}.
This extension is briefly mentioned in \cite{Peherstorfer20}, but not further analyzed.
% =============================================================================
\subsection{Linear Petrov--Galerkin Operator Inference}
\label{sec:pgopinf}
We consider snapshot data \cref{eq:data_matrices} from the time discretization of the state-space system \cref{eq:full_order_state_space_system}. Consequently, the data matrices \cref{eq:data_matrices} are assumed to approximately satisfy the system equations \cref{eq:full_order_state_space_system}, i.e., 
\begin{equation}\label{eq:data_satisfies_system_equations}
	\begin{aligned}
	\DstatesData &= \stateSpaceSystemMatrix \statesData + \stateSpaceInputMatrix\inputData + \residualTermStateTimeInt, \\
	\outputData &= \stateSpaceOutputMatrix \statesData + \stateSpaceFeedthroughMatrix \inputData + \residualTermOutputTimeInt.	
	\end{aligned}
\end{equation}
with discretization or measurement error contributions $\residualTermStateTimeInt\in\R^{\dimStates\times\numDataPoints}$ and $\residualTermOutputTimeInt\in\R^{\dimOutputs\times\numDataPoints}$.
Then, in contrast to the usual \gls{opinf} framework, and in agreement with the general geometric framework for \gls{mor} in \cite{BuchfinkEtAl24}, we project the state data according to the oblique projection \cref{eq:projector_petrov_galerkin} as
\begin{align}
	\begin{aligned}
		\label{eq:projected_data_matrices_petrov_galerkin}
		\reducedStatesData &\vcentcolon= (\reductionBasisLeft\tran\reductionBasis)\inv\reductionBasisLeft\tran\statesData, \\
		\reducedDstatesData &\vcentcolon= (\reductionBasisLeft\tran\reductionBasis)\inv\reductionBasisLeft\tran\DstatesData,
	\end{aligned}
\end{align}
where we assume for now that the test basis $\reductionBasisLeft$ is given.
By setting $\stateSpaceDescriptorMatrixReduced\vcentcolon= \reductionBasisLeft\tran\reductionBasis$ and introducing the notation
\begin{gather*}
	\lhsDstatesOutputDataReduced \vcentcolon= \begin{bmatrix}
		\stateSpaceDescriptorMatrixReduced\reducedDstatesData \\
		\outputData
	\end{bmatrix} \in\R^{(\dimReducedStates+\dimInputs)\times\numDataPoints},\qquad
	\rhsStatesInputDataReduced \vcentcolon= \begin{bmatrix}
		\reducedStatesData \\
		\inputData
	\end{bmatrix} \in\R^{(\dimReducedStates+\dimInputs)\times\numDataPoints},\\
	\blockMatrixStateSpaceABCDReducedIdentified \vcentcolon= \begin{bmatrix}
		\stateSpaceSystemMatrixReducedIdentified & \stateSpaceInputMatrixReducedIdentified \\
		\stateSpaceOutputMatrixReducedIdentified & \stateSpaceFeedthroughMatrixReducedIdentified
	\end{bmatrix} \in\R^{(\dimReducedStates+\dimInputs) \times (\dimReducedStates+\dimInputs)},
\end{gather*}
we obtain the \textit{Petrov--Galerkin \gls{opinf} problem}
\begin{align}
	\label{eq:pgopinf}
	\min_{\blockMatrixStateSpaceABCDReducedIdentified} \|\lhsDstatesOutputDataReduced - \blockMatrixStateSpaceABCDReducedIdentified\rhsStatesInputDataReduced\|_{\indexFrobNorm},
\end{align}
which again can be solved optimally via the pseudoinverse as in \cref{eq:opinf_optimization_problem}.

\subsubsection{Error Analysis of Petrov--Galerkin Operator Inference}

To understand the relation between the intrusive operators defined in~\eqref{eq:reduced_order_state_space_system} and the nonintrusive operators resulting from the Petrov--Galerkin \gls{opinf} problem~\eqref{eq:pgopinf}, we split the state data into its projection onto the reduced space $\reducedSpace$ and the corresponding complement $\orthogonalComplementStatesData\vcentcolon= (\eye{\dimStates}-\projector)\statesData\in\R^{\dimStates\times\numDataPoints}$ of the data with respect to the oblique projection $\projector$ as
\begin{align}
	\label{eq:data_decomposition}
	\statesData &= \reductionBasis\reducedStatesData + \orthogonalComplementStatesData.
\end{align}

\begin{theorem}[\Gls{opinf} decomposition theorem]\label{thm:pgopinf}
	Assume that the data matrices~\cref{eq:data_matrices} satisfy~\cref{eq:data_satisfies_system_equations} for some $\residualTermStateTimeInt\in\R^{\dimStates\times\numDataPoints}$ and $\residualTermOutputTimeInt\in\R^{\dimOutputs\times\numDataPoints}$.
	Moreover, for some $\reductionBasis \in \R^{\dimStates \times \dimReducedStates}$ and $\reductionBasisLeft \in \R^{\dimStates \times \dimReducedStates}$, consider the Petrov--Galerkin \gls{opinf} problem~\cref{eq:pgopinf} with projected data according to~\cref{eq:projected_data_matrices_petrov_galerkin}.
	If the matrix $\rhsStatesInputDataReduced$ has full row rank, then the minimizers of the Petrov--Galerkin \gls{opinf} problem~\cref{eq:pgopinf} satisfy
	\begin{equation}\label{eq:thm:pgopinf}
		\begin{aligned}
			\begin{bmatrix}
			\stateSpaceSystemMatrixReducedIdentifiedMinimizer & \stateSpaceInputMatrixReducedIdentifiedMinimizer\\
			\stateSpaceOutputMatrixReducedIdentifiedMinimizer & \stateSpaceFeedthroughMatrixReducedIdentifiedMinimizer
			\end{bmatrix}
			&= \begin{bmatrix}
				\stateSpaceSystemMatrixReduced & \stateSpaceInputMatrixReduced\\
				\stateSpaceOutputMatrixReduced & \stateSpaceFeedthroughMatrixReduced
			\end{bmatrix} + \begin{bmatrix}
		\reductionBasisLeft\tran\stateSpaceSystemMatrix\orthogonalComplementStatesData + \reductionBasisLeft\tran\residualTermStateTimeInt\\
				\stateSpaceOutputMatrix\orthogonalComplementStatesData + \residualTermOutputTimeInt
			\end{bmatrix}
			\rhsStatesInputDataReduced\pinv,
		\end{aligned}
	\end{equation}
	where $\stateSpaceSystemMatrixReduced$, $\stateSpaceInputMatrixReduced$, $\stateSpaceOutputMatrixReduced$, and $\stateSpaceFeedthroughMatrixReduced$ are the intrusive Petrov--Galerkin reduced operators from \cref{eq:reduced_order_state_space_system}.
\end{theorem}

\begin{proof}
	The assumption that the data matrices \cref{eq:data_matrices} satisfy \cref{eq:data_satisfies_system_equations} implies
	\begin{equation}
		\label{eq:projected_dynamics_data_equation}
		\begin{aligned}
			\reductionBasisLeft\tran\DstatesData &= \reductionBasisLeft\tran\stateSpaceSystemMatrix \statesData + \reductionBasisLeft\tran\stateSpaceInputMatrix\inputData + \reductionBasisLeft\tran\residualTermStateTimeInt\\
			\outputData &= \stateSpaceOutputMatrix \statesData + \stateSpaceFeedthroughMatrix \inputData + \residualTermOutputTimeInt.
		\end{aligned}
	\end{equation}
	Define $\stateSpaceDescriptorMatrixReduced\vcentcolon=\reductionBasisLeft\tran\reductionBasis$, then
	\begin{align}\label{eq:relation_descriptor_to_proj_Dstate}
	 	\reductionBasisLeft\tran\DstatesData =  (\reductionBasisLeft\tran\reductionBasis)(\reductionBasisLeft\tran\reductionBasis)\inv\reductionBasisLeft\tran\DstatesData = \stateSpaceDescriptorMatrixReduced\reducedDstatesData.
	\end{align}
	Hence, together with the data decomposition \cref{eq:data_decomposition}, we can rewrite \cref{eq:projected_dynamics_data_equation} as
	\begin{equation}\label{eq:petrov_galerkin_intrusive_with_residual}
		\begin{aligned}
		\stateSpaceDescriptorMatrixReduced\reducedDstatesData 
			&= \reductionBasisLeft\tran\stateSpaceSystemMatrix\reductionBasis\reducedStatesData 		
			+ \reductionBasisLeft\tran\stateSpaceInputMatrix\inputData
			+ \reductionBasisLeft\tran\stateSpaceSystemMatrix\orthogonalComplementStatesData + \reductionBasisLeft\tran\residualTermStateTimeInt\\
			\outputData &= \stateSpaceOutputMatrix \reductionBasis\reducedStatesData + \stateSpaceFeedthroughMatrix \inputData + \stateSpaceOutputMatrix \orthogonalComplementStatesData + \residualTermOutputTimeInt,
		\end{aligned}
	\end{equation}
	which reads in block form as
	\begin{equation}
		\label{eq:block_form_petrov_galerkin_intrusive_with_residual}
		\begin{bmatrix}
			\stateSpaceDescriptorMatrixReduced\reducedDstatesData \\
			\outputData
		\end{bmatrix}
		=
		\begin{bmatrix}
			\stateSpaceSystemMatrixReduced & \stateSpaceInputMatrixReduced \\
			\stateSpaceOutputMatrixReduced & \stateSpaceFeedthroughMatrixReduced
		\end{bmatrix}
		\begin{bmatrix}
			\reducedStatesData \\
			\inputData
		\end{bmatrix}
		+
		% \underbrace{
		\begin{bmatrix}
			\reductionBasisLeft\tran\stateSpaceSystemMatrix\orthogonalComplementStatesData + \reductionBasisLeft\tran\residualTermStateTimeInt\\
			\stateSpaceOutputMatrix\orthogonalComplementStatesData + \residualTermOutputTimeInt
		\end{bmatrix}.
	\end{equation}
	Finally, a minimizer of the Petrov--Galerkin \gls{opinf} problem~\cref{eq:pgopinf} satisfies
	\begin{align*}
		\begin{bmatrix}
			\stateSpaceSystemMatrixReducedIdentifiedMinimizer & \stateSpaceInputMatrixReducedIdentifiedMinimizer \\
			\stateSpaceOutputMatrixReducedIdentifiedMinimizer & \stateSpaceFeedthroughMatrixReducedIdentifiedMinimizer
		\end{bmatrix} 
		= \lhsDstatesOutputDataReduced\rhsStatesInputDataReduced\pinv
		= \begin{bmatrix}
			\stateSpaceDescriptorMatrixReduced\reducedDstatesData \\
			\outputData
		\end{bmatrix} \rhsStatesInputDataReduced\pinv
		&= \parentheses{
			\begin{bmatrix}
			\stateSpaceSystemMatrixReduced & \stateSpaceInputMatrixReduced \\
			\stateSpaceOutputMatrixReduced & \stateSpaceFeedthroughMatrixReduced
		\end{bmatrix} 
		\rhsStatesInputDataReduced + 
		\begin{bmatrix}
			\reductionBasisLeft\tran\stateSpaceSystemMatrix\orthogonalComplementStatesData + \reductionBasisLeft\tran\residualTermStateTimeInt\\
			\stateSpaceOutputMatrix\orthogonalComplementStatesData + \residualTermOutputTimeInt
		\end{bmatrix}
		}\rhsStatesInputDataReduced\pinv \\
		&=\begin{bmatrix}
			\stateSpaceSystemMatrixReduced & \stateSpaceInputMatrixReduced \\
			\stateSpaceOutputMatrixReduced & \stateSpaceFeedthroughMatrixReduced
		\end{bmatrix} 
		\rhsStatesInputDataReduced\rhsStatesInputDataReduced\pinv + 
		\begin{bmatrix}
			\reductionBasisLeft\tran\stateSpaceSystemMatrix\orthogonalComplementStatesData + \reductionBasisLeft\tran\residualTermStateTimeInt\\
			\stateSpaceOutputMatrix\orthogonalComplementStatesData + \residualTermOutputTimeInt
		\end{bmatrix} \rhsStatesInputDataReduced\pinv
	\end{align*}
	The full row rank assumption for $\rhsStatesInputDataReduced$ implies $\rhsStatesInputDataReduced\rhsStatesInputDataReduced\pinv=\eye{\dimReducedStates+\dimInputs}$, which completes the proof.
\end{proof}

The above decomposition shows that \gls{opinf} can be interpreted as intrusive Petrov–-Galerkin projection augmented by a correction term that accounts for the influence of unresolved state components and time-discretization errors. The first term represents an aliasing effect caused by the projection of the full-order dynamics onto the reduced trial space before the regression step is performed. To this end we introduce the notation
\begin{equation}\label{eq:def_residual_terms}
	\residualTermState \vcentcolon= \reductionBasisLeft\tran\stateSpaceSystemMatrix\orthogonalComplementStatesData + \reductionBasisLeft\tran\residualTermStateTimeInt \qquad\text{and} \qquad 
	\residualTermOutput \vcentcolon= \stateSpaceOutputMatrix\orthogonalComplementStatesData + \residualTermOutputTimeInt.
\end{equation}

\begin{remark}\label{rem:representation_e}
	From~\cref{eq:block_form_petrov_galerkin_intrusive_with_residual} follows that the terms $\residualTermState$ and $\residualTermOutput$ can be equivalently represented as
	\begin{equation*}
		\mathcal{E}_x = \stateSpaceDescriptorMatrixReduced\reducedDstatesData - \stateSpaceSystemMatrixReduced \reducedStatesData - \stateSpaceInputMatrixReduced\inputData
		\quad\text{and}\quad
		\mathcal{E}_y = \outputData - \stateSpaceOutputMatrixReduced\reducedStatesData - \stateSpaceFeedthroughMatrixReduced\inputData.
	\end{equation*}
	This representation for $\residualTermState$ and $\residualTermOutput$ is consistent 
	with the results reported in~\cite{Peherstorfer20}, where the data is obtained via a so-called reprojection algorithm.
	In more detail, for a given $\reductionBasis$, the reprojection algorithm described in \cite[Fig.~2, Alg.~1]{Peherstorfer20} 
	constructs the data iteratively. In each step of the sampling process, the current state is first projected onto 
	the reduced space and directly lifted back to the full space. The lifted state is subsequently used to evaluate the 
	full-order dynamics in order to generate the next state sample for one time step. Under certain conditions, the 
	resulting reprojected trajectories coincide with those obtained from the intrusive \gls{rom}. As a consequence, 
	the error terms $\residualTermState$ and $\residualTermOutput$ vanish, and the operators identified by \gls{opinf} 
	coincide with the intrusive operators of the Petrov--Galerkin projection, which is consistent with \cite[Cor.~3.2]{Peherstorfer20}.
\end{remark}

\subsubsection{Convergence Analysis of Petrov--Galerkin Operator Inference}

\Cref{thm:pgopinf} provides an exact representation of the error between the intrusive and nonintrusive reduced operators. The error terms can be interpreted as follows.
First, if the time-discretization scheme used to construct the data is convergent, then, similarly to \cite[Thm.~1]{PeherstorferWillcox16}, we can expect the discretization errors $\residualTermStateTimeInt$ and $\residualTermOutputTimeInt$ to converge to zero, i.e., the error is dominated by $\reductionBasisLeft\tran\stateSpaceSystemMatrix\orthogonalComplementStatesData\rhsStatesInputDataReduced\pinv$ and $\stateSpaceOutputMatrix\orthogonalComplementStatesData\rhsStatesInputDataReduced\pinv$.
To further analyze these terms, we have to investigate the product $\orthogonalComplementStatesData\rhsStatesInputDataReduced\pinv$ in more detail. Consider the singular value decomposition of $\statesData$ given as
	\begin{align}
	\label{eqn:dataSVD}
	\statesData = \begin{bmatrix}
		\svdUa & \svdUb
	\end{bmatrix}\begin{bmatrix}
		\svdSa & \zeros\\
		\zeros & \svdSb
	\end{bmatrix}\begin{bmatrix}
		\svdVa\tran \\ \svdVb\tran
	\end{bmatrix}, \qquad \begin{aligned}
		\svdSa &= \mathrm{diag}(\sigma_1,\ldots,\sigma_{\dimReducedStates}),\\
		\svdSb &= \begin{bmatrix}\mathrm{diag}(\sigma_{\dimReducedStates+1},\ldots,\sigma_k) & \zeros\\\zeros & \zeros \end{bmatrix},
	\end{aligned}
	\end{align}
	with $[\svdUa,\svdUb]$ and $[\svdVa,\svdVb]$ orthogonal, $k = \rank(\statesData)$, and $\sigma_{\dimReducedStates} > \sigma_{\dimReducedStates+1}$. Let $\reductionBasis=\svdUa$, $\reductionBasisLeft = \reductionBasis$, and define $\inputDataProja \vcentcolon= \inputData \svdVa$ and $\inputDataProjb \vcentcolon= \inputData\svdVb$. We thus compute
\begin{align*}
	\rhsStatesInputDataReduced &= \begin{bmatrix}
		(\reductionBasisLeft^\T \reductionBasis)^{-1} \reductionBasisLeft^\T\statesData\\
		\inputData
	\end{bmatrix} = \begin{bmatrix}
		(\reductionBasisLeft^\T \reductionBasis)^{-1} \reductionBasisLeft^\T \big(\reductionBasis \svdSa \svdVa^\T + \svdUb \svdSb \svdVb^\T\big)\\
		\inputDataProja \svdVa^\T + \inputDataProjb \svdVb^\T
	\end{bmatrix}\\
	&= \begin{bmatrix}
		\svdSa  &  (\reductionBasisLeft^\T \reductionBasis)^{-1} \reductionBasisLeft^\T\svdUb \svdSb\\
		\inputDataProja & \inputDataProjb
	\end{bmatrix}\begin{bmatrix}
		\svdVa^\T\\
		\svdVb^\T
	\end{bmatrix}.
\end{align*}
Note that $\orthogonalComplementStatesData = \svdUb \svdSb \svdVb^\T$  and hence
\begin{align*}
	\orthogonalComplementStatesData\rhsStatesInputDataReduced\pinv = \svdUb\svdSb\begin{bmatrix}
		\zeros & \eye{\numDataPoints-\dimReducedStates}
	\end{bmatrix}\begin{bmatrix}
		\svdSa  &  (\reductionBasisLeft^\T \reductionBasis)^{-1} \reductionBasisLeft^\T\svdUb \svdSb\\
		\inputDataProja & \inputDataProjb
	\end{bmatrix}\pinv.
\end{align*}
In general, computing the pseudoinverse of a block matrix in terms of the original matrix blocks is involved and still an active research area. See for instance~\cite{ZhaMM26} and the references therein for some recent results. However, if we consider a special case of a Galerkin projection, i.e., $\reductionBasisLeft = \reductionBasis$ and use the \gls{pod} basis for the projection, then we obtain, for the single input case, the following result.

\begin{theorem}
	\label{thm:OpInfNoConv}
	Assume a single input, i.e., $\dimInputs = 1$. 
	If $\reductionBasisLeft = \reductionBasis=\svdUa$ (as in \cref{eqn:dataSVD}) and $\sigma_{\dimReducedStates} > \sigma_{\dimReducedStates+1}$, then
	\begin{align}
		\label{eqn:GalerkinOIerrorBound}
		\frobNorm*{\orthogonalComplementStatesData\rhsStatesInputDataReduced\pinv} \leq \frac{\sigma_{\dimReducedStates+1}}{\sigma_{\dimReducedStates}} \frobNorm*{\inputDataProjb\pinv} \sqrt{\|\inputDataProja\|_{\text{F}}^2 + \dimInputs\sigma_{\dimReducedStates}^2}.
	\end{align}
\end{theorem}

\begin{proof}
	Using $\reductionBasisLeft = \svdUa$, we obtain
	\begin{align*}
	\begin{bmatrix}
		\svdSa  &  (\reductionBasisLeft^\T \reductionBasis)^{-1} \reductionBasisLeft^\T\svdUb \svdSb\\
		\inputDataProja & \inputDataProjb
	\end{bmatrix} = \begin{bmatrix}
		\svdSa  &  \zeros\\
		\inputDataProja & \inputDataProjb
	\end{bmatrix}
	\end{align*}
	We distinguish the cases $\inputDataProjb=\zeros$ and $\inputDataProjb\neq\zeros$. First, let us consider the case $\inputDataProjb=\zeros$. Then $\inputDataProjb\pinv = \zeros$ and
	\begin{align*}
	\begin{bmatrix}
		\zeros & \eye{\numDataPoints-\dimReducedStates}
	\end{bmatrix}\begin{bmatrix}
		\svdSa  &  \zeros\\
		\inputDataProja & \zeros
	\end{bmatrix}\pinv = \zeros,
	\end{align*}
	such that~\eqref{eqn:GalerkinOIerrorBound} is trivially satisfied. On the other hand, if $\inputDataProjb\neq \zeros$, then $\mathrm{rank}(\inputDataProjb) = 1 = \dimInputs$. Hence, the Banachiewicz--Schur form for the pseudoinverse (see~\cite[Cor.\,1]{Mia91} or \cite[Cor.~2]{BakS02}) implies
	\begin{align*}
		\orthogonalComplementStatesData\rhsStatesInputDataReduced\pinv &= \svdUb\svdSb\begin{bmatrix}
		\zeros & \eye{\numDataPoints-\dimReducedStates}
	\end{bmatrix}
	\begin{bmatrix}
		\svdSa^{-1} & \zeros\\
		-\inputDataProjb\pinv\inputDataProja \svdSa^{-1} & \inputDataProjb\pinv
	\end{bmatrix} = \svdUb \svdSb \inputDataProjb\pinv \begin{bmatrix}
		-\inputDataProja \svdSa^{-1} & \eye{\dimInputs}
	\end{bmatrix}.
	\end{align*}
	We conclude, using standard results for the sub-multiplicativity of the Frobenius norm and the spectral norm, i.e., for arbitrary matrices $\arbitraryMatrixA$ and $\arbitraryMatrixB$ of compatible dimensions,
	\begin{equation*}
		\frobNorm{\arbitraryMatrixA\arbitraryMatrixB} \leq \twoNorm{\arbitraryMatrixA} \frobNorm{\arbitraryMatrixB} \quad\text{and}\quad \frobNorm{\arbitraryMatrixA\arbitraryMatrixB} \leq \frobNorm{\arbitraryMatrixA} \twoNorm{\arbitraryMatrixB},
	\end{equation*}
	that
	\begin{align*}
		\frobNorm{\orthogonalComplementStatesData\rhsStatesInputDataReduced\pinv} 
		& \leq \|\svdSb\|_{2} \|\inputDataProjb\pinv\|_{\text{F}} \|\begin{bmatrix}
		-\inputDataProja \svdSa^{-1} & \eye{\dimInputs}
		\end{bmatrix}\|_{\text{F}}
		= \|\svdSb\|_{2} \|\inputDataProjb\pinv\|_{\text{F}} \sqrt{\|\inputDataProja \svdSa^{-1}\|_{\text{F}}^2 + \dimInputs} \\
		&\leq \|\svdSb\|_{2} \|\inputDataProjb\pinv\|_{\text{F}}\sqrt{\|\inputDataProja\|_{\text{F}}^2  \|\svdSa^{-1}\|_{2}^2 + \dimInputs}
		= \frac{\sigma_{\dimReducedStates+1}}{\sigma_{\dimReducedStates}}
		\|\inputDataProjb\pinv\|_{\text{F}} \sqrt{\|\inputDataProja\|_{\text{F}}^2 + \dimInputs \sigma_{\dimReducedStates}^{2}}.\qedhere
	\end{align*}
\end{proof}

We make the following observations.
\begin{enumerate}
	\label{en:zero-inp}
	% U_2 = 0
	\item If $\inputDataProjb = \zeros$, which is for instance the case if $\inputData = \zeros$, i.e., the case of autonomous systems, then $\orthogonalComplementStatesData\rhsStatesInputDataReduced\pinv = \zeros$. Assuming that~$\reducedStatesData$ has full row rank, we thus conclude
	\begin{align*}
	\left\|\begin{bmatrix}
		\stateSpaceSystemMatrixReducedIdentifiedMinimizer\\\stateSpaceOutputMatrixReducedIdentifiedMinimizer
	\end{bmatrix} - \begin{bmatrix}
		\stateSpaceSystemMatrixReduced\\\stateSpaceOutputMatrixReduced
	\end{bmatrix}\right\|_{\text{F}}\to 0\qquad\text{for } \left\|\begin{bmatrix}\residualTermStateTimeInt\\\residualTermOutputTimeInt\end{bmatrix}\right\|_{\text{F}}\to 0,
	\end{align*}
	which is an even stronger result than \cite[Thm.~1]{PeherstorferWillcox16}, as it does not require the reduced dimension to approach the full dimension.
	% U \neq 0
	\item If $\inputData \neq \zeros$, convergence of $\frobNorm*{\orthogonalComplementStatesData\rhsStatesInputDataReduced\pinv}$ to $0$ is not immediately guaranteed from~\eqref{eqn:GalerkinOIerrorBound}. In particular, since we often have $\sigma_{\dimReducedStates+1} / \sigma_{\dimReducedStates} \approx \alpha$ for some $\alpha\in(0,1)$, there is no inherent reason for
    \begin{align*}
        \big\|\inputDataProjb\pinv\big\|_{\text{F}} \sqrt{\|\inputDataProja\|_{\text{F}}^2 + \dimInputs \sigma_{\dimReducedStates}^2} \geq \big\|\inputDataProjb\pinv\big\|_{\text{F}} \|\inputDataProja\|_{\text{F}}
    \end{align*}
    to decay, except in the limiting case $\dimReducedStates = \dimStates$.
\end{enumerate}

\Cref{thm:OpInfNoConv} therefore suggests that, except in special cases such as $\inputData=\zeros$ or $r=n$, Galerkin \gls{opinf} cannot in general be expected to recover the operators obtained by intrusive projection. While \Cref{thm:OpInfNoConv} focuses on the case $\reductionBasisLeft = \reductionBasis$, analogous limitations are expected when $\reductionBasisLeft \neq \reductionBasis$.

\begin{remark}
	The restriction to scalar inputs, i.e., $\dimInputs = 1$ is primarily technical and arises from the pseudoinverse block formula used in the proof. Similar bounds are expected to hold for multiple inputs, although the proof in its current form would require additional range assumptions and are out of the scope of this paper.
\end{remark}

\subsection{Port-Hamiltonian Petrov--Galerkin Operator Inference}
\label{sec:method:linear_ph_identification}
The inference problem of a linear reduced \gls{ph} system in its most general form can be cast as a constrained optimization problem in block matrix form as
\begin{equation}
    \label{eq:general_linear_reduced_ph_optimization_problem}
    \begin{aligned}
        \min_{\phStructureMatrixReducedIdentified,\phGeneralDissipationMatrixReducedIdentified,\phEReducedIdentified,\phQReducedIdentified}\ &\norm{\begin{bmatrix}
            \phEReducedIdentified\reducedDstatesData \\ -\outputData
        \end{bmatrix} - 
        \parentheses{\phStructureMatrixReducedIdentified - \phGeneralDissipationMatrixReducedIdentified}
        \begin{bmatrix} 
            \phQReducedIdentified\reducedStatesData \\ \inputData 
        \end{bmatrix}}_{F}, \\    
        \text{subject to}\quad
        &\phStructureMatrixReducedIdentified = -\phStructureMatrixReducedIdentified\tran,\\ 
        &\phGeneralDissipationMatrixReducedIdentified = \phGeneralDissipationMatrixReducedIdentified\tran \succeq 0,\\ 
        &\phEReducedIdentified\tran \phQReducedIdentified = \phQReducedIdentified\tran \phEReducedIdentified \succ 0,
    \end{aligned}    
\end{equation}
with
\begin{align}
\phStructureMatrixReducedIdentified := 
\begin{bmatrix}
    \phJReducedIdentified & \phInputMatrixReducedIdentified \\
    -\phInputMatrixReducedIdentified\tran & \phFeedthroughOutputMatrixSkewReducedIdentified
\end{bmatrix}, 
\quad 
\phGeneralDissipationMatrixReducedIdentified := 
\begin{bmatrix}
    \phRReducedIdentified & \phFeedthroughInputMatrixReducedIdentified \\
    \phFeedthroughInputMatrixReducedIdentified\tran & \phFeedthroughOutputMatrixSymReducedIdentified
\end{bmatrix},
\end{align}
where all \gls{ph} matrices are the identified counterparts, defined analogously to those in \cref{eq:pg:intrusive:rom:ph}.
The optimization problem \cref{eq:general_linear_reduced_ph_optimization_problem} constitutes a constrained, nonlinear, and nonconvex problem, which is generally difficult to solve.

In the following, we pursue two complementary routes to address this challenge.
In \cref{sec:method:ph_petrov_galerkin_opinf}, we extend the Petrov--Galerkin \gls{opinf} framework from \cref{sec:pgopinf} to \gls{ph} systems under the assumption that the Hessian of the Hamiltonian $\phQ$ is known.
To relax this requirement, \cref{sec:method:hamiltonian_based_phQ_identification} augments this approach with a data-driven procedure for identifying $\phQ$ from Hamiltonian data.
As an alternative route, \cref{sec:method:ph_convex_optimization} presents a one-shot convex formulation that does not require prior knowledge of the Hamiltonian.

%%%%%%%%%%%%%%%%%%%%%%%%%%%%%%%%%%%%%%%%
\subsubsection{Port-Hamiltonian Petrov--Galerkin Operator Inference}
\label{sec:method:ph_petrov_galerkin_opinf}
%%%%%%%%%%%%%%%%%%%%%%%%%%%%%%%%%%%%%%%%
For the moment, we assume to have access to the Hessian $\phQ$ of the Hamiltonian \cref{eq:hamiltonian_quadratic_form}, which has also been done in \cite{MorandinEtAl23}. 
Let us review the Petrov--Galerkin \gls{opinf} approach from \cref{sec:pgopinf} in the context of inferring reduced \gls{ph} systems.
From \cref{eq:pg:intrusive:rom:ph:sys} we have seen that an intrusive reduced \gls{ph} system can be obtained through the Petrov--Galerkin projection of a full-order \gls{ph} system with a test basis $\reductionBasisLeft=\phQ\reductionBasis$.
We now insert this choice of test basis into the projection of the data as in \cref{eq:projected_data_matrices_petrov_galerkin}, i.e.,
\begin{equation}
	\label{eq:projected_data_matrices_petrov_galerkin_ph}
	\begin{aligned}
		\reducedStatesData &\vcentcolon= (\reductionBasis\tran\phQ\reductionBasis)\inv\reductionBasis\tran\phQ\statesData, &
		\reducedDstatesData &\vcentcolon= (\reductionBasis\tran\phQ\reductionBasis)\inv\reductionBasis\tran\phQ\DstatesData,
	\end{aligned}	
\end{equation}
and further solve the Petrov--Galerkin \gls{opinf} problem \cref{eq:pgopinf} with the projected data including $\phEReduced=\reductionBasisLeft\tran\reductionBasis=\reductionBasis\tran\phQ\tran\reductionBasis$.

We obtain identified reduced operators $\pHStateSpaceSystemMatrixReducedIdentified\in\R^{\dimReducedStates\times\dimReducedStates}$, $\pHStateSpaceInputMatrixReducedIdentified\in\R^{\dimReducedStates\times\dimInputs}$, $\pHStateSpaceOutputMatrixReducedIdentified\in\R^{\dimInputs\times\dimReducedStates}$, and $\pHStateSpaceFeedthroughMatrixReducedIdentified\in\R^{\dimInputs\times\dimInputs}$ that converge to
\begin{align}
	% A and B
	\begin{bmatrix}
		\pHStateSpaceSystemMatrixReducedIdentified & \pHStateSpaceInputMatrixReducedIdentified\\
		\pHStateSpaceOutputMatrixReducedIdentified & \pHStateSpaceFeedthroughMatrixReducedIdentified
	\end{bmatrix} = \begin{bmatrix}
		\phJReduced - \phRReduced & \phInputMatrixReduced - \phFeedthroughInputMatrixReduced\\
		\phInputMatrixReduced + \phFeedthroughInputMatrixReduced & \phFeedthroughOutputMatrixSymReduced + \phFeedthroughOutputMatrixSkewReduced
	\end{bmatrix} + \begin{bmatrix}
		\reductionBasis\tran\phQ\tran(\phJ-\phR)\phQ\orthogonalComplementStatesData \rhsStatesInputDataReduced\pinv \\
		(\phInputMatrix + \phFeedthroughInputMatrix)\tran
		\end{bmatrix}\orthogonalComplementStatesData \rhsStatesInputDataReduced\pinv,
\end{align}
if the original data has been calculated via a convergent time integration scheme of a full-order \gls{ph} system \cref{eq:linear_ph_system_with_feedthrough} and the data matrices have full column rank.
Here, the residual term is calculated with the \gls{ph} projector $\projectorPH\vcentcolon=\reductionBasis(\reductionBasis\tran\phQ\reductionBasis)\inv\reductionBasis\tran\phQ\in\R^{\dimStates\times\dimStates}$ as $\orthogonalComplementStatesData = (\eye{\dimStates} - \projectorPH)\statesData$.

% %%%%%%%%%%%%%%%%%%%%%%%%%%%%%%%%%%%%%%%%
\subsubsection{Inferring the Hessian of the Hamiltonian}
\label{sec:method:hamiltonian_based_phQ_identification}
%%%%%%%%%%%%%%%%%%%%%%%%%%%%%%%%%%%%%%%%%%
Having access to $\phQ$ from \cref{eq:hamiltonian_quadratic_form} might not always be possible.
In certain cases, the internal energy $\hamiltonian(\cdot)$ or an approximation of it may be accessible from simulations or measurements.
For example, many commercial finite element software packages allow for the extraction of energy quantities.
Motivated by this, we propose to estimate the Hessian of the Hamiltonian matrix directly from data of the Hamiltonian; cf.~\cite{Nicodemus25}.
This leads to the optimization problem
\begin{equation}\label{eq:hamiltonian_based_phQ_identification:opt}
	\min_{\phQIdentifiedHamiltonian = \phQIdentifiedHamiltonian\tran \succ 0} \sum_{i=1}^{\numDataPoints} \norm{\hamiltonian(\states(\timee_i)) - \tfrac{1}{2}\states_i\tran\phQIdentifiedHamiltonian\states_i}_{\indexFrobNorm}.
\end{equation}

To rewrite~\eqref{eq:hamiltonian_based_phQ_identification:opt} into a linear least-squares problem, we will rely on Kronecker algebra. Due to the inherent symmetry of the quadratic form, the Kronecker product includes duplicated terms (e.g., $\state_1\state_2$ and $\state_2\state_1$). These can be removed using the duplication matrix $\duplicateMat \in \R^{\dimStates^2 \times \frac{\dimStates(\dimStates+1)}{2}}$ and the half-vectorization operation $\mathrm{vech}\colon\R^{\dimStates\times\dimStates}\to\R^{ \frac{\dimStates(\dimStates+1)}{2}}$; cf.~\cite{MagnusNeudecker80}.
Ignoring, for the moment, the positive semidefinite constraint, the optimization problem~\eqref{eq:hamiltonian_based_phQ_identification:opt} can be recast as a linear least-squares problem of the form
\begin{align}
    \label{eq:hamiltonian_identification_least_square_problem}
	\min_{\vech{\phQIdentifiedHamiltonian}}
	\Biggl\|
	\begin{bmatrix}
        \hamiltonian(\states(\timee_1)) \\ \vdots \\ \hamiltonian(\states(\timee_{\numDataPoints}))
    \end{bmatrix} -
    \tfrac{1}{2}\underbrace{\begin{bmatrix}
        \states_1 \kron \states_1 \\
        \vdots \\
        \states_{\numDataPoints} \kron \states_{\numDataPoints}
    \end{bmatrix}}_{\statesDataKron}
    \duplicateMat
    \vech{\phQIdentifiedHamiltonian}\Biggr\|_2,
\end{align}
where $\kron$ denotes the Kronecker product, resulting in the Kronecker data matrix $\statesDataKron \in \R^{\numDataPoints \times \dimStates^2}$.
Note that \cref{eq:hamiltonian_identification_least_square_problem} requires at least $\numDataPoints\geq\frac{\dimStates(\dimStates+1)}{2}$ independent data points for a unique solution.

Numerical experiments indicate that when the number of data points is insufficient, the resulting matrix $\phQIdentifiedHamiltonian$ is typically indefinite, exhibiting negative eigenvalues of non-negligible magnitude. In such cases, one may project onto the cone of positive definite matrices; however, this can yield a poor approximation of the true Hessian of the Hamiltonian $\phQ$ due to the induced projection error. Therefore, exploiting the fact that~\cref{eq:hamiltonian_based_phQ_identification:opt} defines a convex optimization problem, we propose, in the small-data regime, to instead solve the problem directly with a positive definiteness constraint via semidefinite programming.

%%%%%%%%%%%%%%%%%%%%%%%%%%%%%%%%%%%%%%%%
\subsubsection{Port-Hamiltonian Inference Through Convex Optimization}
\label{sec:method:ph_convex_optimization}
%%%%%%%%%%%%%%%%%%%%%%%%%%%%%%%%%%%%%%%%
As mentioned, the inference of \gls{ph} systems using Petrov--Galerkin \gls{opinf} requires knowledge of the Hessian of the Hamiltonian $\phQ$ (\cref{sec:method:ph_petrov_galerkin_opinf}) or data from the Hamiltonian (\cref{sec:method:hamiltonian_based_phQ_identification}).
The \gls{ph} properties originate from the assumption that the residual term becomes small and we converge towards the intrusive operators.

In \cref{sec:intro:literature:ph_system_identification}, we surveyed existing approaches to tackle the challenging \gls{ph} optimization problem \cref{eq:general_linear_reduced_ph_optimization_problem}.
These range from solving the problem with nonlinear optimization techniques or reformulating it into an unconstrained problem (for the state equation) as
\begin{align}
    \label{eq:unconstrained_optimization_problem}
    \min \norm{\DstatesData - (\unconstrainedPHJ-\unconstrainedPHJ\tran-\unconstrainedPHR\unconstrainedPHR\tran)\unconstrainedPHQ\unconstrainedPHQ\tran\statesData}
\end{align}
which increases the nonconvexity by introducing parameter cross-terms \cite{GoyPB25,GoyalEtAl23a}.
In this section, we pursue another approach that directly incorporates the constraints on the matrices and does not require any knowledge of $\phQ$.

The first important observation is that in \cref{eq:general_linear_reduced_ph_optimization_problem}, the term $\phQReducedIdentified$ in combination with the operators $\phJReducedIdentified$ and $\phRReducedIdentified$ leads to nonconvexity.
In \cite{MehrmannUnger23} it is described how to remove the $\phQ$ from the \gls{fom} \cref{eq:linear_ph_system_with_feedthrough} by simply multiplying the equation by $\phQ\tran$ on the left, which leads
\begin{subequations}
\label{eq:linear_berlin_ph_system_with_feedthrough}
\begin{equation}
\label{eq:linear_berlin_ph_system_with_feedthrough:sys}
\begin{aligned}
\phQ\tran\Dstates(\timee) &= (\phJBerlin-\phRBerlin)\states(\timee) + (\phInputMatrixBerlin - \phFeedthroughInputMatrixBerlin) \phInputs(\timee),
\quad \states(0) = \initialConditionStates, \\
\phOutputs(\timee) &= (\phInputMatrixBerlin + \phFeedthroughInputMatrixBerlin)\tran\states(\timee) + (\phFeedthroughOutputMatrixSymBerlin + \phFeedthroughOutputMatrixSkewBerlin)\phInputs(\timee),
\end{aligned}
\end{equation}
where the transformed matrices are defined as
\begin{equation}\label{eq:linear_berlin_ph_system_with_feedthrough:mats}
	\begin{aligned}
    \phJBerlin &:= \phQ\tran\phJ\phQ,
    &\phRBerlin &:= \phQ\tran\phR\phQ,
    &\phInputMatrixBerlin &:= \phQ\tran\phInputMatrix,\\
    \phFeedthroughInputMatrixBerlin &:= \phQ\tran\phFeedthroughInputMatrix,
    &\phFeedthroughOutputMatrixSymBerlin &:= \phFeedthroughOutputMatrixSym,
    &\phFeedthroughOutputMatrixSkewBerlin &:= \phFeedthroughOutputMatrixSkew.
	\end{aligned}
\end{equation}
\end{subequations}
Performing a Galerkin projection on the transformed system \cref{eq:linear_berlin_ph_system_with_feedthrough} leads to the same reduced system as in \cref{eq:pg:intrusive:rom:ph}.
From the optimization problem \cref{eq:general_linear_reduced_ph_optimization_problem}, we can now eliminate $\phQReducedIdentified$ using the transformation introduced above.
We are left with
\begin{equation}
    \label{eq:convex_linear_reduced_ph_optimization_problem}
    \begin{aligned}
        \min_{\phStructureMatrixReducedIdentified,\phGeneralDissipationMatrixReducedIdentified,\phEReducedIdentified}\ &\norm{\begin{bmatrix}
            \phEReducedIdentified\reducedDstatesData \\ -\outputData
        \end{bmatrix} - 
        \parentheses{\phStructureMatrixReducedIdentified - \phGeneralDissipationMatrixReducedIdentified}
        \begin{bmatrix} 
            \reducedStatesData \\ \inputData 
        \end{bmatrix}}_{\indexFrobNorm}, \\    
        \text{subject to}\quad
        &\phStructureMatrixReducedIdentified = -\phStructureMatrixReducedIdentified\tran,\\ 
        &\phGeneralDissipationMatrixReducedIdentified = \phGeneralDissipationMatrixReducedIdentified\tran \succeq 0,\\ 
        &\phEReducedIdentified\tran = \phEReducedIdentified \succ 0.
    \end{aligned}
\end{equation}

\begin{proposition}
	The optimization problem~\cref{eq:convex_linear_reduced_ph_optimization_problem} is a convex optimization problem.
\end{proposition}

\begin{proof}
	Let us define the residual function $\residualOptimization$ of the optimization problem as
\begin{align}
	\residualOptimization(\phEReducedIdentified,\phStructureMatrixReducedIdentified,\phGeneralDissipationMatrixReducedIdentified) := \begin{bmatrix}
            \phEReducedIdentified\reducedDstatesData \\ -\outputData
        \end{bmatrix} - 
        \parentheses{\phStructureMatrixReducedIdentified - \phGeneralDissipationMatrixReducedIdentified}
        \begin{bmatrix} 
            \reducedStatesData \\ \inputData 
        \end{bmatrix}
\end{align}
and observe that $\residualOptimization$ is affine in its arguments.
Hence, the residual $\residualOptimization$ is affine in all decision variables.
Moreover, \cref{eq:convex_linear_reduced_ph_optimization_problem} is convex, since
the Frobenius norm is convex and the composition of a convex function with an affine function is again convex.
Finally, the constraints on the matrices are convex sets: the set of skew-symmetric matrices is a linear subspace, the set of positive semidefinite matrices is a convex cone, and the set of positive definite matrices is an open convex cone.
Consequently, the entire optimization problem \cref{eq:convex_linear_reduced_ph_optimization_problem} is a convex optimization problem.
\end{proof}

The optimization problem~\cref{eq:convex_linear_reduced_ph_optimization_problem} is a semidefinite program for which various efficient solvers exist; see \cite{GillEtAl19}.
In practical implementations, the strict positive definiteness constraint $\phEReducedIdentified \succ 0$ is typically enforced numerically by requiring $\phEReducedIdentified \succeq \regularizationTermPosDef \eye{\dimStates}$ with a small $\regularizationTermPosDef>0$, which ensures compatibility with interior-point solvers while preserving convexity.
Note that since the Hessian of the Hamiltonian is not known a priori, we use the standard projection strategy and a \gls{pod}-based reduction basis in this identification setting.

% =============================================================================
\section{Numerical Experiments}\label{sec:numerics}
% =============================================================================

We illustrate the theoretical findings and the proposed methodologies in several numerical experiments on well-established benchmark problems.
% Methodologies in use
In our investigation, we will compare the following methodologies:
\begin{itemize}
	\item \gpod: the intrusive Galerkin \gls{rom} obtained by projecting the \gls{fom} with a \gls{pod} basis $\reductionBasis$ as in \cref{eq:reduced_order_state_space_system}, where $\reductionBasisLeft = \reductionBasis$;
	\item \gopinf~\cite{PeherstorferWillcox16}: solving \cref{eq:pgopinf} with Galerkin-projected data as in~\cref{eq:projected_data_matrices_galerkin};
	\item \pgpod: the intrusive Petrov--Galerkin reduced system obtained by projecting the \gls{fom} system with a \gls{pod} basis $\reductionBasis$ and the test basis $\reductionBasisLeft = \phQ\reductionBasis$ as in \cref{eq:projected_data_matrices_petrov_galerkin};
	\item \pgopinf: solving \cref{eq:pgopinf} with obliquely projected data as in \cref{eq:projected_data_matrices_petrov_galerkin} where $\reductionBasisLeft = \phQ\reductionBasis$ with known Hessian of the Hamiltonian $\phQ$ from the corresponding \gls{ph} \gls{fom}~\cref{eq:linear_ph_system_with_feedthrough};
	\item \pgopinfham: solving \cref{eq:pgopinf} with obliquely projected data as in \cref{eq:projected_data_matrices_petrov_galerkin} where $\reductionBasisLeft = \phQIdentifiedHamiltonian\reductionBasis$ with the Hessian of the Hamiltonian $\phQIdentifiedHamiltonian$ is identified from Hamiltonian data via solving \cref{eq:hamiltonian_based_phQ_identification:opt} with semidefinite programming;	
	\item \cvxph: solving \cref{eq:convex_linear_reduced_ph_optimization_problem} via semidefinite programming without knowledge of the Hessian of the Hamiltonian $\phQ$ and using standard Galerkin-projected data from \cref{eq:projected_data_matrices_galerkin}.
\end{itemize}

Throughout the experiments, we employ a convergent time discretization and obtain the derivative data $\DstatesData$ subsequently via evaluating the \gls{fom} dynamics. As a consequence, we obtain $\residualTermStateTimeInt=\residualTermOutputTimeInt=0$ in \cref{eq:data_satisfies_system_equations}.

We consider four benchmark examples: \begin{enumerate*}
	\item the CD player system from \cite{DraSB92} \label{en:benchmarks:cd-player}
	\item the earth atmosphere model from \cite{FarI95} \label{en:benchmarks:earth-atmosphere}
	\item the mass-spring-damper system from \cite{GugercinEtAl12} \label{en:benchmarks:msd}
	\item and the poroelasticity system from \cite{AMU21}. \label{en:benchmarks:poro}
\end{enumerate*}
For each example, we proceed similarly: we simulate the system via the implicit midpoint rule and use a frequency sweep as input signal.
As starting frequency $\freqmin$, we choose a value that is sufficiently small to excite the low-frequency dynamics, and as end frequency $\freqmax$, we choose a value that is sufficiently large to excite the high-frequency dynamics. An educated guess for these values can be obtained by looking at the singular values of the transfer function for a range of frequencies (singular value plot).
Then, in order to prevent anti-alising effects, we choose the time step $\Delta \timee=\frac{1}{2\freqmax}$ and set the final time at least to $\timeeEnd \ge 10 \sqrt{\freqmin \freqmax}$, which corresponds to ten times the geometric mean of the starting and end frequencies. 
From the simulated data, we choose every $k$th snapshot to form the data matrices $\statesData$, $\DstatesData$, $\inputData$, and $\outputData$.
We proceed with the methods and report values for $\dimReducedStates=4,12,\dots, 80$. The numerical parameters, including the state dimension $\dimStates$ and input/output dimension $\dimInputs=\dimOutputs$ and the rank of the data matrix $\rank(\statesData)$ for the examples are summarized in \cref{tab:1}.

\begin{table}[ht]
  	\centering
  	\caption{Numerical parameters for the benchmark examples.}
  	\label{tab:1}
	\resizebox{0.98\linewidth}{!}{
	\begin{tabular}{l l c c S S S c c c}
		\toprule
		\multicolumn{2}{l}{Example} & {$\dimStates$} & {$\dimInputs$} & {$\freqmin$} & {$\freqmax$} & {$\Delta t$} & {$\timeeEnd$} & {$k$} & {$\rank(\statesData)$}\\
		\midrule
		\ref{en:benchmarks:cd-player} & CD-Player & 120 & 2 & 1 & 1e5 & 5e-6 & \num{0.06324555320336758} & 10 & 117 \\
		\ref{en:benchmarks:earth-atmosphere} & Earth atmosphere & 598 & 1 & 1e-2 & 1e1 & 5e-2 & \num{31.62277660168379} & 1 & 122 \\
		\ref{en:benchmarks:msd} & Mass-spring-damper & 100 & 2 & 1e-2 & 1 & 5e-1 & 200 & 3,4 & 85 \\
		\ref{en:benchmarks:poro} & Poroelasticity & 980 & 2 & 1e-1 & 1e3 & 5e-4 & 2 & 2 & 271\\
		\bottomrule
	\end{tabular}
	}
\end{table}

% Goals of the experiment with some error introduction
The experiments serve multiple purposes: first, we verify the theoretical results from \cref{sec:pgopinf} regarding the convergence of the identified operators to the intrusive operators in \cref{sec:numerics:1}.
Then, we compare \gopinf with the \pgopinf in terms of accuracy and stability of the \glspl{rom} in \cref{sec:numerics:2}.
Finally, we investigate in \cref{sec:numerics:3} the influence of the choice of the Hessian of the Hamiltonian $\phQ$ in the test basis $\reductionBasisLeft$ and compare it with the one-shot approach from \cref{sec:method:ph_convex_optimization}.

Throughout this section, the $\Hinf$-norm computations are carried out using pyMOR~\cite{MRS16}, whereas the convex optimization problem \cref{eq:convex_linear_reduced_ph_optimization_problem} is solved using CVXPY\footnote{\url{https://www.cvxpy.org/}} with MOSEK\footnote{\url{https://docs.mosek.com/11.0/pythonapi/index.html}}.

\begin{data_code_information_box}    
    The code used to generate the subsequent results is accessible via
		\begin{center}
			\href{https://github.com/Institute-Eng-and-Comp-Mechanics-UStgt/pgopinf}{https://github.com/Institute-Eng-and-Comp-Mechanics-UStgt/pgopinf}.
		\end{center}
  	%with a persistent release in~\todo{Upload package to Zenodo once it will be published} under MIT Common License.
\end{data_code_information_box}

\subsection{Distance to Intrusive Operators and Error Bounds}\label{sec:numerics:1}
We start with showcasing the theoretical findings from \cref{thm:pgopinf} and \cref{thm:OpInfNoConv} by comparing identified operators via \gopinf with the intrusive Galerkin projected operators (\gpod). 
To this end, we report the following quantities related to the results from \cref{thm:pgopinf} and \cref{thm:OpInfNoConv}:
\begin{equation*}
	\begin{aligned}
		\errOps &\vcentcolon= \frobNorm{\begin{bmatrix}
			\stateSpaceSystemMatrixReducedIdentifiedMinimizer & \stateSpaceInputMatrixReducedIdentifiedMinimizer\\
			\stateSpaceOutputMatrixReducedIdentifiedMinimizer & \stateSpaceFeedthroughMatrixReducedIdentifiedMinimizer
			\end{bmatrix}
			- \begin{bmatrix}
				\stateSpaceSystemMatrixReduced & \stateSpaceInputMatrixReduced\\
				\stateSpaceOutputMatrixReduced & \stateSpaceFeedthroughMatrixReduced
			\end{bmatrix}},
			&
		\normCorr &\vcentcolon= \frobNorm{
		\begin{bmatrix}
		\reductionBasisLeft\tran\stateSpaceSystemMatrix\orthogonalComplementStatesData\\
				\stateSpaceOutputMatrix\orthogonalComplementStatesData
		\end{bmatrix}
		\rhsStatesInputDataReduced\pinv
		},\\
		\errEst & \vcentcolon= 
			\frac{\sigma_{\dimReducedStates+1}}{\sigma_{\dimReducedStates}} \big\|\inputDataProjb\pinv\big\|_{\text{F}} \sqrt{\|\inputDataProja\|_{\text{F}}^2 + \dimInputs\sigma_{\dimReducedStates}^2}, \qquad
		&
		\errBound &\vcentcolon= \errEst \frobNorm{
			\begin{bmatrix}
				\reductionBasisLeft\tran\stateSpaceSystemMatrix\\
				\stateSpaceOutputMatrix
			\end{bmatrix}}.
	\end{aligned}
\end{equation*}
In more detail, $\varepsilon_{\text{ref}}$ denotes the distance between the identified operators and the intrusively projected operators. $\varepsilon_{\text{corr}}$ is the norm of the correction term in \cref{thm:pgopinf}. The quantity $\varepsilon_{\text{est}}$ is the bound given in \cref{thm:OpInfNoConv}; however, it only serves as an estimator since it bounds only $\frobNorm{\orthogonalComplementStatesData\rhsStatesInputDataReduced\pinv}$. 
Finally, $\varepsilon_{\text{bound}}$ denotes the full bound obtained by again using the submultiplicativity of the Frobenius norm.

The results are depicted in \cref{fig:1}. Additional to the quantities above, we also report the singular values $\sigma$ of the state data matrix $\statesData$.

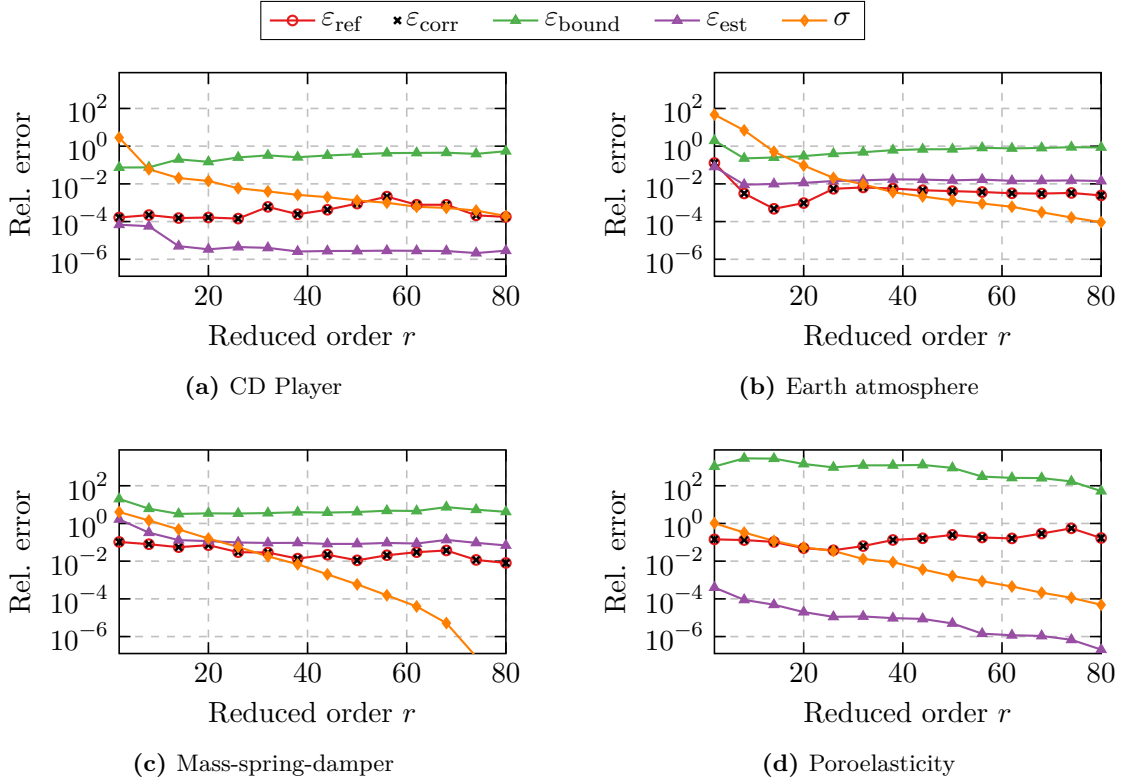
\begin{figure}[htb]
	\centering
	\ref{leg:cdplayer}\\
	\vspace{1em}
	\begin{subfigure}{0.475\linewidth}
		\centering
		\input{plots/cdplayer_opinf_metrics.tikz}
		\vspace{-2.5ex} 
		\caption{CD Player}
		\label{fig:1:a}
	\end{subfigure}
	\hfill
	\begin{subfigure}{0.475\linewidth}
		\centering
		\input{plots/earth_atmosphere_opinf_metrics.tikz}
		\vspace{-2.5ex} 
		\caption{Earth atmosphere}
		\label{fig:1:b}
	\end{subfigure}\\[1.5em]
	\begin{subfigure}{0.475\linewidth}
		\centering
		\input{plots/msd_opinf_metrics.tikz}
		\vspace{-2.5ex} 
		\caption{Mass-spring-damper}
		\label{fig:1:c}
	\end{subfigure}
	\hfill
	\begin{subfigure}{0.475\linewidth}
		\centering
		\input{plots/poro_opinf_metrics.tikz}
		\vspace{-2.5ex} 
		\caption{Poroelasticity}
		\label{fig:1:d}
	\end{subfigure}
	\caption{Relative values of the quantities related to the \cref{thm:pgopinf} and \cref{thm:OpInfNoConv} and singular values $\sigma$ of the state data matrix $\statesData$ for several benchmark examples, where the \glspl{rom} are obtained by \gpod and \gopinf.}
	\label{fig:1}
\end{figure}

As expected, $\errOps = \normCorr$ for all reduced dimensions $\dimReducedStates$ in all examples.
Although the singular values $\sigma$ decay, there is no convergence rate of $\errOps$ observed.
This behavior is consistent with the theoretical findings in \cref{thm:OpInfNoConv}, since $\errEst$ (resp. $\errBound$) remain nearly constant, as expected from the discussion after \cref{thm:OpInfNoConv}.

We also confirmed the theoretical result in~\cref{en:zero-inp}: for all examples above, using a constant zero input signal yielded a relative error between the state and output matrices on the order of machine precision. To obtain non-trivial state trajectories in this setting, the time simulations were initialized with a non-zero initial condition.

\subsection{Stability-encouraging Property of PG-OpInf}\label{sec:numerics:2}
In this section, we turn our attention to accuracy and stability of the \glspl{rom} obtained by the \pgopinf. Therefore, we report the $\Hinf$-error, which is defined as 
\begin{equation*}
	\HinfNorm{\transferFunction} \vcentcolon= \sup_{\laplaceVariable \in \Cplus} \twoNorm{\transferFunction(\laplaceVariable)},
\end{equation*}
where $\transferFunction:\Cplus\to\C^{\dimOutputs \times \dimInputs}$ is transfer function of the error system. 
We revisit the \gls{ph} examples, namely, the mass-spring-damper and poroelasticity systems, and assume throughout this section that the Hessian of the Hamiltonian, $\phQ$, is known as it is used for data reduction as in \eqref{eq:projected_data_matrices_petrov_galerkin_ph}, and for the intrusive projection as in~\eqref{eq:pg:intrusive:rom:ph}.
The remaining setup is identical to that in \cref{sec:numerics:1}.

In \cref{fig:2}, we report the $\Hinf$ error for varying reduced orders of the resulting \glspl{rom}, obtained using \gpod, \gopinf, \pgpod, and \pgopinf. 

\begin{figure}[htb]
	\centering
	\ref{leg:msd:hinf}\\
	\vspace{1em}
	\begin{subfigure}{0.475\linewidth}
		\centering
		\input{plots/msd_hinf_exp2.tikz}
		\vspace{-2.5ex} 
		\caption{Mass-spring-damper}
		\label{fig:2:a}
	\end{subfigure}\hfill
	\begin{subfigure}{0.475\linewidth}
		\centering
		\input{plots/poro_hinf_exp2.tikz}
		\vspace{-2.5ex} 
		\caption{Poroelasticity}
		\label{fig:2:b}
	\end{subfigure}
	\caption{$\Hinf$ error for the mass-spring-damper and poroelasticity examples for reduced orders $\dimReducedStates = 2, 8, \ldots, 80$, using \glspl{rom} obtained via \gpod, \gopinf, \pgpod, and \pgopinf. Dashed lines indicate cases where the original \gls{rom} is unstable; in these instances, the $\Hinf$ error of the stable component is reported instead.}
	\label{fig:2}
\end{figure}
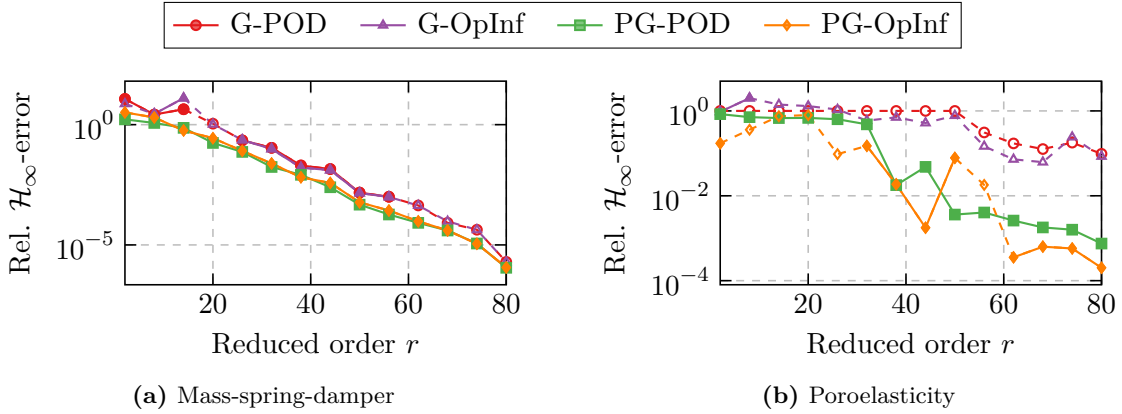

We observe that, for certain reduced orders in both examples, the \glspl{rom} obtained via \gopinf, \gpod, or \pgopinf are unstable, resulting in an infinite $\Hinf$ error. 
For these cases, we perform a stable-unstable decomposition~\cite{BCS16} and report the $\Hinf$ error of the stable component, which is shown as dashed lines in \cref{fig:2}.
In more detail, for the mass-spring-damper example, we observe that for reduced orders $\dimReducedStates=20,62,68,74,80$, the \glspl{rom} obtained by \gopinf and \gpod are unstable. 
In contrast, the \glspl{rom} obtained by \pgopinf (and \pgpod) are stable for all reduced orders and show also a smaller $\Hinf$-error for most reduced orders compared to the Galerkin approaches.

For the poroelasticity example, most \glspl{rom} obtained via the Galerkin approaches are unstable, whereas for \pgopinf only the reduced orders $\dimReducedStates = 8, 14, 20, 26, 56$ yield unstable models. As in the mass-spring-damper example, the Petrov--Galerkin approaches generally exhibit smaller $\Hinf$ errors than the Galerkin approaches. This difference becomes more pronounced for higher reduced orders, where the Petrov--Galerkin methods achieve $\Hinf$ errors that are several orders of magnitude lower.

To further illustrate the instability of the \glspl{rom} obtained via \gopinf and \gpod, we report in \cref{tab:2} the spectral abscissa of the reduced system matrix $\stateSpaceSystemMatrixReduced$ for the poroelasticity example, which is defined as
$\alpha(\stateSpaceSystemMatrixReduced) = \max\{\real{\eigenvalue} : \eigenvalue \in \sigma(\stateSpaceSystemMatrixReduced)\}$,
where $\sigma(\cdot)$ denotes the spectrum of a matrix. A positive spectral abscissa indicates instability, whereas a negative value indicates stability.

\begin{table}[ht]
  \centering
  \caption{Spectral abscissa for the poroelasticity example; stable models are highlighted in green.}\label{tab:2}
  \begin{tabular}{l S S[negative-color=LimeGreen] S[negative-color=LimeGreen] S[negative-color=LimeGreen]}
    \toprule
    {$\dimReducedStates$} & {\gpod} & {\gopinf} & {\pgpod} & {\pgopinf} \\
    \midrule
    \begin{filecontents*}{data/poro_spectral_abscissa_cmp_over_r.csv}
r,pod_opinf_g | identified | False,pod_opinf_g | intrusive | False,pod_opinf_g | identified | True,pod_opinf_g | intrusive | True,pod_opinf_pg | identified | False,pod_opinf_pg | intrusive | False,pod_opinf_pg | identified | True,pod_opinf_pg | intrusive | True
2,-21.45391487140712,2.165758963754969,-21.453914871407115,,-103.67710052386226,-59.58564929589088,-103.67710052386228,-59.585649295890846
8,-8.689980788264684,49.532289337383965,-8.689980788264629,-5.381886801499623,296.3997267566191,-38.4346487708811,-58.159666059105874,-38.4346487708815
14,852.3078895247369,946.0464696752758,-5.935263123865283,,77.92827944049182,-68.79550771124313,-27.528803157801317,-68.79550771124225
20,81.5885717787564,133.32637020578886,-4.488346300623593,-0.4505560579547285,231.5802828717695,-71.28729029734988,-0.8145028311268506,-71.28729029734885
26,121.74502339137864,190.8200496476404,-5.569421413990368,-5.650600624492357,30.42622641204951,-64.62286029616949,-55.676082923703575,-64.6228602961691
32,177.5582406837817,192.9300895167299,-6.600563356170977,-8.626350912328213,-1.4442486591860269,-92.61947575154689,-1.4442486591855055,-92.61947575154926
38,383.0447464351926,421.8911865664117,-1.8421769437095428,-0.3236059480793226,-1.9838689501247049,-1.8718345640303011,-1.9838689501231817,-1.871834564030307
44,12.001135913849712,9.34370171268133,-1.8421769389538865,-0.3027972391109379,-1.8723898790012965,-1.8718341770616227,-1.8723898790013007,-1.8718341770616316
50,240.5785080443105,221.87727491305108,-1.6474868182752076,-0.298811428280154,-0.5974687152378572,-1.664713852881242,-0.5974687158039176,-1.6647138528812506
56,160.75078219671195,217.57161683431707,-1.6474777432712857,-0.2461902024373315,79.19265635475601,-1.6535977626707894,-1.6591957744746997,-1.6535977627731813
62,142.38526573809685,187.09479914575004,-1.6474774586273693,-0.8180357653838775,-1.654277857013155,-1.6535871937573308,-1.6542778570132175,-1.6535871937704008
68,74.15053586349623,153.62867312534195,-1.6282297274231403,-0.5794277502832776,-1.6292987764540583,-1.629194857019099,-1.6292987764539264,-1.62919485701913
74,165.2822967941031,170.68911758823643,-1.6263635721746574,-1.6263640513526418,-1.6275966184583643,-1.6264986363973877,-1.627596618459454,-1.626498636397179
80,1000.768159873822,962.7650775344852,-1.626363552263305,-1.6263640263578505,-1.6265038008981858,-1.6264925158267425,-1.626503800898791,-1.6264925158249697
86,1342.1903954868903,1354.5216403179509,-1.6262233990366743,-1.6262234632766543,-1.6263017222771745,-1.6262626608801838,-1.6263017223067675,-1.626262660879279
92,3902.136721914348,4062.4625130389854,-1.6262040016624388,-0.4775682642253266,-1.6262364839161627,-1.6262375802163112,-1.6262364839130503,-1.6262375802161242
98,3945.7916867447375,4093.612723663989,-0.7881719092236473,-0.367782192984935,-1.6262547265127092,-1.6262416987701058,-1.6262547265004887,-1.6262416988556954
104,3929.634171597416,4072.932320723684,-0.3503131029356369,-0.3568015210904004,101.01211611488884,-1.2412693922269715,-1.626241886502629,-1.2412693963302202
110,3373.5145835128437,3377.7888438899167,-0.0582943007154305,-0.0714553220500303,79.03365478301487,-0.3122243976041713,-0.2021921711032948,-0.3122244051463833
116,3342.920065196839,3344.9112083787586,-0.0685896667382264,-0.0768553260598575,149.22749308120038,-0.1687520381996435,-0.2391206351601555,-0.1687520381947885
122,3259.25485982104,3259.0233050912493,-0.0885727993373861,-0.0892725007508734,-0.1683278999963401,-0.1678408595258558,-0.1683278999667453,-0.1678408592890627
128,2416.425855873548,2415.536969012467,-0.0817990958243705,-0.0825321885885765,-0.1679187775780422,-0.1677904075263778,-0.1679187778124733,-0.1677904075196847
134,2323.267083692604,2322.0482773751137,-0.0823477179517579,-0.08324181362829,0.0740561258433336,-0.1677683350175756,-0.1677711468058065,-0.1677683349259065
140,2054.6063881636924,2048.344729449024,-0.0804467215943441,-0.0716113981143891,0.0392109307422378,-0.1677508572795808,-0.1677599136248162,-0.1677508575187283
146,1963.3440640231468,1961.2499382953133,-0.0818207407323374,-0.0817250464672736,7.689663796528886,-0.167719722516358,-0.1677701003271144,-0.167719720787397
152,1902.20162541567,1902.196202755046,-0.0818907251158527,-0.0817773804734838,-0.1651770302549519,-0.1677205034131816,-0.1651770293656893,-0.1677205034946068
158,1906.8560346016611,1906.831012834804,-0.0814628839749256,-0.081497356702943,-0.1587478836275005,-0.1677108506256676,-0.1587478843954354,-0.1677108508075294
164,2083.34469969478,2083.814521524489,-0.081283972822935,-0.0812533227229083,-0.1382891914325148,-0.1677024651530259,-0.1382892149460618,-0.1677024648828221
170,2299.302664323792,2300.315609028763,-0.0816772893069282,-0.0815556698912368,-0.1658456453637125,-0.1677040156025023,-0.1658456462331368,-0.1677040154935651
176,2168.6817631284466,2166.4009656024573,-0.0811027841820251,-0.0811759394634859,-0.1670656550446633,-0.167711999391845,-0.1670656544107261,-0.1677119995198375
182,1721.056538988102,1734.9724781431712,-0.0811676714951186,-0.0811742938376107,-0.1669235231122482,-0.1677104835596672,-0.1669235232853169,-0.1677104834588126
188,1591.5741504913717,1610.4193138898167,-0.0811983444518234,-0.0812005882780149,-0.1675273012990955,-0.1677033310303405,-0.167527301506396,-0.1677033311161106
194,1599.7313619873682,1602.7049336677503,-0.0206956845137622,-0.0811991450475484,-0.1676875456569564,-0.1524085721173023,-0.1676875456785449,-0.1524300260210706
200,1588.0288465537317,1587.8215296641338,-0.0812002198635516,-0.0812001505648501,-0.1676820702872589,-0.1471663182382345,-0.167682070346772,-0.1471625873023268
\end{filecontents*}
\csvreader[
      late after line  = \\,   % newline after every data row
      late after last line = \\\bottomrule,
	  filter test = {\ifnumcomp{\r}{<}{81}}
    ]{data/poro_spectral_abscissa_cmp_over_r.csv}%
    {r=\r,pod_opinf_g | intrusive | False=\podg,pod_opinf_g | identified | False=\opinf,pod_opinf_pg | intrusive | False=\podpg,pod_opinf_pg | identified | False=\opinfpg}%
	{\r & \podg & \opinf & \podpg & \opinfpg}
  \end{tabular}
\end{table}

We observe that the \glspl{rom} from \gopinf (and \gpod) have unstable poles, which are not of small magnitude and are thus no numerical artifacts.
For the reported reduced orders, the \glspl{rom} obtained by \pgopinf (and \pgpod) are almost always more stable than their Galerkin counterparts in the sense that the spectral abscissa is smaller than the spectral abscissa of the Galerkin counterparts.
As demonstrated in \cref{fig:2} and \cref{tab:2}, \pgopinf does not guarantee stability of the resulting \glspl{rom}; however, it substantially increases the likelihood of obtaining stable models compared to the Galerkin approaches.
In more detail, for the poroelasticity example, 68\% of the \glspl{rom} with reduced orders $\dimReducedStates = 2, 8, \ldots, 200$ obtained via \pgopinf are stable, whereas only 6\% of those obtained via \gopinf are stable. 
Similarly, for the mass-spring-damper example, 97.5\% of the \glspl{rom} computed with \pgopinf are stable, compared to 60\% for \gopinf over reduced orders $\dimReducedStates = 2, 4, \ldots, 80$.

\subsection{Inferring the Hamiltonian and the pH One-shot Identification}\label{sec:numerics:3}
In our final example, we revisit the mass-spring-damper system, now considering the case where the Hessian of the Hamiltonian, $\phQ$, is unknown. We first apply the strategy from \cref{sec:method:hamiltonian_based_phQ_identification}, assuming that measurements of the Hamiltonian are available to estimate $\phQIdentifiedHamiltonian$ a priori. Alternatively, we employ the one-shot \cvxph operator inference approach from \cref{sec:method:ph_convex_optimization}.

For the exact setting of \cref{sec:numerics:2}, the estimation of~$\phQIdentifiedHamiltonian$ fails to converge, yielding a~$\phQ$ with entries of large magnitude ($\sim 10^6$), and consequently the \pgopinfham approach performs poorly. However, by increasing the number of data points, i.e., reducing the sampling rate to $k \delta \timee = 1.5$, corresponding to every third time-discretized data point, we obtain a reasonable estimate of $\phQIdentifiedHamiltonian$.

With this setup, we apply only the non-intrusive methods (\gopinf, \pgopinf, \pgopinfham, and \cvxph).
We emphasize that, like \gopinf, both \pgopinfham and \cvxph do not have access to the Hamiltonian Hessian $\phQ$. While \pgopinfham first estimates $\phQIdentifiedHamiltonian$ from Hamiltonian data and then uses it for the projection (as in \cref{sec:method:hamiltonian_based_phQ_identification}), \cvxph directly identifies the reduced operators by solving the convex optimization problem~\eqref{eq:convex_linear_reduced_ph_optimization_problem}, without any knowledge of $\phQ$.

We report the $\Hinf$ errors again in \cref{fig:msd_hinf_exp3} and observe that all methods produce stable \glspl{rom} for all reduced orders. As in \cref{sec:numerics:2}, \pgopinf (with known $\phQ$) achieves for most reduced orders the smallest $\Hinf$ errors, while \pgopinfham and \cvxph yield errors comparable to \gopinf.

We conclude, that \pgopinf demonstrates strong potential. However, in practical applications when the Hamiltonian Hessian is unknown, our current approach to determine~$\reductionBasisLeft$ is not yet fully satisfactory. 
Further research is needed to improve the estimation of~$\phQ$ or to identify alternative choices for $\reductionBasisLeft$, which remains an active area of investigation; cf.~\cite{EPB26, OttPR22}.

\begin{figure}[htb]
	\centering
	\ref{leg:msd:exp3:hinf}\\
	\vspace{1em}
	\input{plots/msd_hinf_exp3.tikz}
	\caption{$\Hinf$-error for the mass-spring-damper example for reduced orders $\dimReducedStates=2,8,\ldots,80$ for \glspl{rom} obtained by \gopinf, \pgopinf, \pgopinfham, and \cvxph, where the latter two methods do not have access to the Hessian of the Hamiltonian $\phQ$.
	}
	\label{fig:msd_hinf_exp3}
\end{figure}
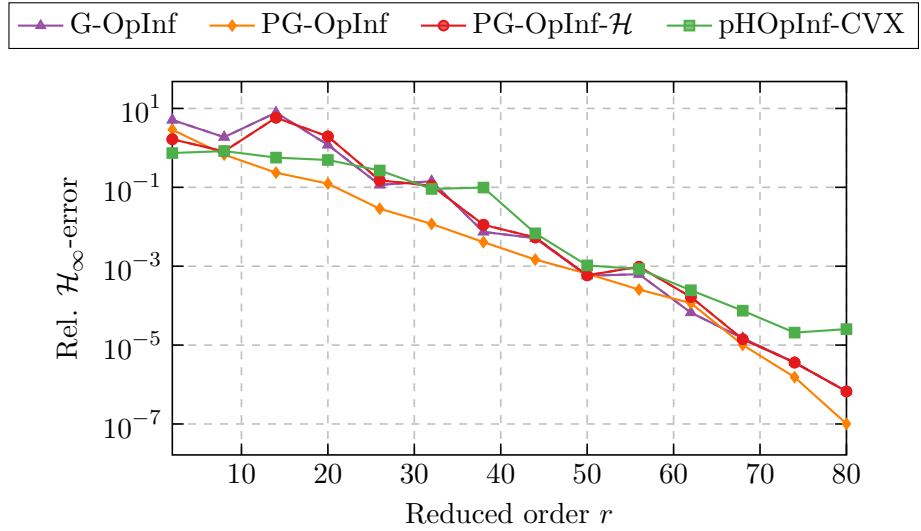

% =============================================================================
\section{Conclusions}\label{sec:conclusions}
% =============================================================================
This work makes several contributions to structured data-driven modeling for \gls{lti} systems.
First, we developed and analyzed a Petrov--Galerkin variant of \gls{opinf} (PG-OpInf) and established its connection to intrusive Petrov--Galerkin model reduction.
In particular, we derived an explicit characterization of the relation between the identified and intrusively projected operators and showed that, except in special settings such as constant zero input signals, the identified operators do not generally converge to their intrusive counterparts.
However, we presented numerical examples that still show that \pgopinf is stability encouraging when a specific choice of the test basis is made, which is motivated by structure-preserving model reduction of \gls{ph} systems.
To address the requirement of the Hessian of the Hamiltonian in the structure-preserving Petrov--Galerkin projection, we propose a data-driven framework for inferring this matrix from measurements of the Hamiltonian or alternatively, introduced a novel one shot operator inference approach that guarantees \gls{ph} structure based on a constrained convex optimization problem.
The effectiveness and theoretical validity of the proposed methods are validated through numerical experiments on several benchmark examples.

% -----------------------------------------------------------------------------------
\subsection*{Acknowledgments}
% -----------------------------------------------------------------------------------
JF, JN, and JR acknowledge funding from the DFG under Germany's Excellence Strategy - EXC 2075 - 390740016, and are thankful for support by the Stuttgart Center for Simulation Science (SimTech).
JR thanks the Department of Mechanical and Aerospace Engineering at UC San Diego for hosting him on a research stay, which was also funded by SimTech.
BU acknowledges funding by the Deutsche Forschungsgemeinschaft (DFG, German Research Foundation) -- Project-ID 258734477 -- SFB 1173.
BK was supported by the U.S. Office of Naval Research under award number N00014-22-1-2624 and the U.S. National Science Foundation's  Civil, Mechanical, and Manufacturing Innovation division under award number 2144023.

% -----------------------------------------------------------------------------------
\subsection*{Tool and Computational Resource Disclosure}
% -----------------------------------------------------------------------------------
Generative AI tools were used to support the preparation of this work in several ways. For the mathematical analysis, they were used to independently double-check selected derivations, results, and mathematical claims. In the software development, they were used for code completion and, in particular, for suggestions concerning the organization and structure of the accompanying software package. For the manuscript, generative AI tools were used for language editing and to improve the clarity and flow of individual formulations. All AI-assisted output was critically reviewed and verified by the authors, who take full responsibility for the scientific content of the work.

%-----------------------------------------------------------------------------%
\printbibliography

\end{document}

%% file: variables.tex
\DeclarePairedDelimiter\abs{\lvert}{\rvert}% |x|
\DeclarePairedDelimiter\norm{\lVert}{\rVert}% ||x||
\DeclarePairedDelimiter\set{\lbrace}{\rbrace} % {x \mid x} use \mid for set properties
\DeclarePairedDelimiter\innerprod{\langle}{\rangle} % <x,y> use , to separate components
\DeclarePairedDelimiter\parentheses{(}{)} % parentheses that change size depending on content
\DeclarePairedDelimiter\brackets{[}{]} % brackets that change size depending on content
\makeatletter
\let\oldabs\abs
\def\abs{\@ifstar{\oldabs}{\oldabs*}}
\let\oldnorm\norm
\def\norm{\@ifstar{\oldnorm}{\oldnorm*}}
\let\oldinnerprod\innerprod
\def\innerprod{\@ifstar{\oldinnerprod}{\oldinnerprod*}}
\let\oldset\set
\def\set{\@ifstar{\oldset}{\oldset*}}
\let\oldparentheses\parentheses
\def\parentheses{\@ifstar{\oldparentheses}{\oldparentheses*}}
\let\oldbrackets\brackets
\def\brackets{\@ifstar{\oldbrackets}{\oldbrackets*}}
\makeatother

\newcommand{\cvxph}{pHOpInf-CVX\xspace}
\newcommand{\gopinf}{G-OpInf\xspace}
\newcommand{\pgopinf}{PG-OpInf\xspace}
\newcommand{\pgopinfham}{PG-OpInf-$\hamiltonian$\xspace}

\newcommand{\gpod}{G-POD\xspace}
\newcommand{\pgpod}{PG-POD\xspace}

\DeclareMathOperator{\rank}{rank}
\DeclareMathOperator{\Real}{Re}   % real part of complex number
\newcommand{\real}[1]{\Real\parentheses{#1}}
\newcommand{\R}{\ensuremath\mathbb{R}}

\newcommand{\C}{\ensuremath\mathbb{C}}
\newcommand{\Cplus}{\C^{+}}

\newcommand{\N}{\ensuremath\mathbb{N}}

\newcommand{\Spd}[1]{\mathcal{S}^{#1}_{\succ}}
\newcommand{\Spsd}[1]{\mathcal{S}^{#1}_{\succeq}}

\newcommand{\T}{\intercal}
\newcommand{\tran}{^{\intercal}} % transpose
\newcommand{\inv}{^{-1}} % inverse
\newcommand{\pinv}{^{\dagger}} % pseudo-inverse
\newcommand{\zeros}{\bm{0}}
\newcommand{\eye}[1]{\bm{I}_{#1}}
\newcommand{\Spann}{\operatorname{span}}
\newcommand{\Span}[1]{\Spann\parentheses{#1}}
\newcommand{\kron}{\otimes}                          % kronecker product
\newcommand{\vech}[1]{\operatorname{vech}(#1)}       % half-vectorization
\newcommand{\indexFrobNorm}{\mathrm{F}}
\NewDocumentCommand{\frobNorm}{s m}{%
    \IfBooleanTF{#1}
        {\norm*{#2}_{\indexFrobNorm}}
        {\norm{#2}_{\indexFrobNorm}}%
}
\newcommand{\twoNorm}[1]{\norm{#1}_{2}}

\newcommand{\Hinf}{\mathcal{H}_\infty}
\newcommand{\HinfNorm}[1]{\norm{#1}_{\Hinf}}

\newcommand{\redIndex}{\mathrm{r}}

\newcommand{\reduce}[1]{#1_{\redIndex}}
\newcommand{\reducePH}[1]{#1_{\redIndex,\mathrm{pH}}}
\newcommand{\berlinAccent}[1]{\widehat{#1}}
\newcommand{\identifiedAccent}[1]{\tilde{#1}}
\newcommand{\unconstrainedAccent}[1]{\breve{#1}}

\newcommand{\timee}{t}
\newcommand{\timeeEnd}{T}

\newcommand{\timeInterval}{\mathbb{T}}
\newcommand{\ddt}{\ensuremath{\frac{\mathrm{d}}{\mathrm{d}t}}}

\newcommand{\stateSpaceSystemMatrix}{\bm{A}}
\newcommand{\stateSpaceInputMatrix}{\bm{B}}
\newcommand{\stateSpaceOutputMatrix}{\bm{C}}
\newcommand{\stateSpaceFeedthroughMatrix}{\bm{D}}
\newcommand{\stateSpaceDescriptorMatrix}{\bm{E}}
\newcommand{\stateSpaceInput}{u}
\newcommand{\stateSpaceInputs}{\bm{\stateSpaceInput}}
\newcommand{\stateSpaceOutput}{y}
\newcommand{\stateSpaceOutputs}{\bm{\stateSpaceOutput}}

\newcommand{\dimInputs}{m}
\newcommand{\dimOutputs}{p}
\newcommand{\dimStates}{n}
\newcommand{\transferFunction}{\mathcal{G}}
\newcommand{\laplaceVariable}{s}
\newcommand{\spectrum}{\sigma}

\newcommand{\eigenvalue}{\lambda}

\newcommand{\storageFunction}{\hamiltonian}
\newcommand{\stateSpaceSystemMatrixReduced}{\reduce{\stateSpaceSystemMatrix}}
\newcommand{\stateSpaceInputMatrixReduced}{\reduce{\stateSpaceInputMatrix}}
\newcommand{\stateSpaceOutputMatrixReduced}{\reduce{\stateSpaceOutputMatrix}}
\newcommand{\stateSpaceFeedthroughMatrixReduced}{\reduce{\stateSpaceFeedthroughMatrix}}
\newcommand{\stateSpaceDescriptorMatrixReduced}{\reduce{\stateSpaceDescriptorMatrix}}
\newcommand{\stateSpaceOutputsReduced}{\reduce{\stateSpaceOutputs}} % use same as pH
\newcommand{\blockMatrixStateSpaceABCD}{\bm{M}}

\newcommand{\blockMatrixStateSpaceABCDReducedIdentified}{\reduce{\identifiedAccent{\blockMatrixStateSpaceABCD}}}

\newcommand{\iteratorTimeDiscretization}{i}
\newcommand{\stateSpaceSystemMatrixReducedIdentified}{\reduce{\identifiedAccent{\stateSpaceSystemMatrix}}}
\newcommand{\stateSpaceInputMatrixReducedIdentified}{\reduce{\identifiedAccent{\stateSpaceInputMatrix}}}
\newcommand{\stateSpaceOutputMatrixReducedIdentified}{\reduce{\identifiedAccent{\stateSpaceOutputMatrix}}}
\newcommand{\stateSpaceFeedthroughMatrixReducedIdentified}{\reduce{\identifiedAccent{\stateSpaceFeedthroughMatrix}}}

\newcommand{\state}{x}
\newcommand{\states}{\bm{\state}}
\newcommand{\initialConditionStates}{\states_{0}}
\newcommand{\Dstates}{\dot{\states}}
\newcommand{\hamiltonian}{\mathcal{H}}
\newcommand{\phJ}{\bm{J}}
\newcommand{\phR}{\bm{R}}
\newcommand{\phQ}{\bm{Q}}
\newcommand{\phE}{\bm{E}}
\newcommand{\phInputMatrix}{\bm{G}}
\newcommand{\phFeedthroughInputMatrix}{\bm{P}}
\newcommand{\phFeedthroughOutputMatrixSym}{\bm{S}}
\newcommand{\phFeedthroughOutputMatrixSkew}{\bm{N}}
\newcommand{\phGeneralDissipationMatrix}{\bm{\Xi}}
\newcommand{\phStructureMatrix}{\bm{\Gamma}}
\newcommand{\phInput}{\stateSpaceInput} % same as for state-space
\newcommand{\phInputs}{\bm{\phInput}}
\newcommand{\phOutput}{\stateSpaceOutput}
\newcommand{\phOutputs}{\bm{\phOutput}}
\newcommand{\phJBerlin}{\berlinAccent{\phJ}}
\newcommand{\phRBerlin}{\berlinAccent{\phR}}

\newcommand{\phInputMatrixBerlin}{\berlinAccent{\phInputMatrix}}
\newcommand{\phFeedthroughInputMatrixBerlin}{\berlinAccent{\phFeedthroughInputMatrix}}
\newcommand{\phFeedthroughOutputMatrixSymBerlin}{\berlinAccent{\phFeedthroughOutputMatrixSym}}
\newcommand{\phFeedthroughOutputMatrixSkewBerlin}{\berlinAccent{\phFeedthroughOutputMatrixSkew}}

\newcommand{\statesData}{\bm{X}}
\newcommand{\DstatesData}{\dot{\bm{X}}}
\newcommand{\inputData}{\bm{U}}
\newcommand{\outputData}{\bm{Y}}
\newcommand{\rhsStatesInputData}{\bm{T}}

\newcommand{\lhsDstatesOutputData}{\bm{Z}}

\newcommand{\numDataPoints}{N}
\newcommand{\phQIdentifiedHamiltonian}{\identifiedAccent{\phQ}_{\hamiltonian}}

\newcommand{\phGeneralDissipationMatrixReducedIdentified}{\reduce{\identifiedAccent{\phGeneralDissipationMatrix}}}

\newcommand{\phStructureMatrixReducedIdentified}{\reduce{\identifiedAccent{\phStructureMatrix}}}
\newcommand{\phEReducedIdentified}{\reduce{\identifiedAccent{\phE}}}
\newcommand{\phQReducedIdentified}{\reduce{\identifiedAccent{\phQ}}}
\newcommand{\phJReducedIdentified}{\reduce{\identifiedAccent{\phJ}}}
\newcommand{\phRReducedIdentified}{\reduce{\identifiedAccent{\phR}}}

\newcommand{\phInputMatrixReducedIdentified}{\reduce{\identifiedAccent{\phInputMatrix}}}
\newcommand{\phFeedthroughInputMatrixReducedIdentified}{\reduce{\identifiedAccent{\phFeedthroughInputMatrix}}}
\newcommand{\phFeedthroughOutputMatrixSymReducedIdentified}{\reduce{\identifiedAccent{\phFeedthroughOutputMatrixSym}}}
\newcommand{\phFeedthroughOutputMatrixSkewReducedIdentified}{\reduce{\identifiedAccent{\phFeedthroughOutputMatrixSkew}}}

\newcommand{\freqmin}{f_1}
\newcommand{\freqmax}{f_2}

\newcommand{\pHStateSpaceSystemMatrixReducedIdentified}{\reducePH{\identifiedAccent{\stateSpaceSystemMatrix}}}
\newcommand{\pHStateSpaceInputMatrixReducedIdentified}{\reducePH{\identifiedAccent{\stateSpaceInputMatrix}}}
\newcommand{\pHStateSpaceOutputMatrixReducedIdentified}{\reducePH{\identifiedAccent{\stateSpaceOutputMatrix}}}
\newcommand{\pHStateSpaceFeedthroughMatrixReducedIdentified}{\reducePH{\identifiedAccent{\stateSpaceFeedthroughMatrix}}}

\newcommand{\unconstrainedPHJ}{\unconstrainedAccent{\phJ}}
\newcommand{\unconstrainedPHQ}{\unconstrainedAccent{\phQ}}
\newcommand{\unconstrainedPHR}{\unconstrainedAccent{\phR}}

\newcommand{\reductionBasis}{\bm{V}}
\newcommand{\reducedSpace}{\mathcal{V}}
\newcommand{\reductionBasisLeft}{\bm{W}}
\newcommand{\reducedSpaceTest}{\mathcal{W}}
\newcommand{\reducedStates}{\states_{\redIndex}}
\newcommand{\reducedDStates}{\dot{\states}_{\redIndex}}
\newcommand{\DreducedStates}{\reducedDStates}
\newcommand{\reducedphOutputs}{\phOutputs_{\redIndex}}

\newcommand{\dimReducedStates}{r}
\newcommand{\phJReduced}{\reduce{\phJ}}
\newcommand{\phRReduced}{\reduce{\phR}}
\newcommand{\phQReduced}{\reduce{\phQ}}
\newcommand{\phEReduced}{\reduce{\phE}}

\newcommand{\phInputMatrixReduced}{\reduce{\phInputMatrix}}
\newcommand{\phFeedthroughInputMatrixReduced}{\reduce{\phFeedthroughInputMatrix}}
\newcommand{\phFeedthroughOutputMatrixSymReduced}{\reduce{\phFeedthroughOutputMatrixSym}}
\newcommand{\phFeedthroughOutputMatrixSkewReduced}{\reduce{\phFeedthroughOutputMatrixSkew}}

\newcommand{\rhsStatesInputDataReduced}{\reduce{\rhsStatesInputData}}
\newcommand{\lhsDstatesOutputDataReduced}{\reduce{\lhsDstatesOutputData}}

\newcommand{\projector}{\Pi}

\newcommand{\reducedStatesData}{\reduce{\statesData}}
\newcommand{\reducedDstatesData}{\reduce{\DstatesData}}
\newcommand{\projectorPH}{\Pi_{\mathrm{pH}}}

\newcommand{\duplicateMat}{\bm{D}_{n}}
\newcommand{\statesDataKron}{\statesData_{\kron}}

\newcommand{\residualOptimization}{\mathcal{R}}
\newcommand{\regularizationTermPosDef}{\epsilon}

\newcommand{\orthogonalComplementStatesData}{\statesData_{\perp}}
\newcommand{\residualTerm}{\mathcal{E}}
\newcommand{\residualTermState}{\residualTerm_{\states}}
\newcommand{\residualTermOutput}{\residualTerm_{\phOutputs}}
\newcommand{\residualTermStateTimeInt}{\Delta_{\states}}
\newcommand{\residualTermOutputTimeInt}{\Delta_{\phOutputs}}
\newcommand{\minimizerAccent}[1]{#1^{*}}
\newcommand{\stateSpaceSystemMatrixReducedIdentifiedMinimizer}{\minimizerAccent{\stateSpaceSystemMatrixReducedIdentified}}
\newcommand{\stateSpaceInputMatrixReducedIdentifiedMinimizer}{\minimizerAccent{\stateSpaceInputMatrixReducedIdentified}}  
\newcommand{\stateSpaceOutputMatrixReducedIdentifiedMinimizer}{\minimizerAccent{\stateSpaceOutputMatrixReducedIdentified}}
\newcommand{\stateSpaceFeedthroughMatrixReducedIdentifiedMinimizer}{\minimizerAccent{\stateSpaceFeedthroughMatrixReducedIdentified}}
\newcommand{\blockMatrixStateSpaceABCDReducedIdentifiedMinimizer}{\minimizerAccent{\reduce{\identifiedAccent{\blockMatrixStateSpaceABCD}}}}

\newcommand{\svdUa}{\bm{\Theta}_1}
\newcommand{\svdUb}{\bm{\Theta}_2}
\newcommand{\svdSa}{\bm{\Sigma}_1}
\newcommand{\svdSb}{\bm{\Sigma}_2}
\newcommand{\svdVa}{\bm{\Psi}_1}
\newcommand{\svdVb}{\bm{\Psi}_2}
\newcommand{\inputDataProja}{\inputData_1}
\newcommand{\inputDataProjb}{\inputData_2}

\newcommand{\errOps}{\varepsilon_{\text{ref}}}
\newcommand{\normCorr}{\varepsilon_{\text{corr}}}
\newcommand{\errEst}{\varepsilon_{\text{est}}}
\newcommand{\errBound}{\varepsilon_{\text{bound}}}

\newcommand{\arbitraryMatrixA}{\bm{A}}
\newcommand{\arbitraryMatrixB}{\bm{B}}

%% file: glossaries.tex
\usepackage[acronym]{glossaries}
\glsdisablehyper % disable hyperlinks
\newacronym{ph}{pH}{port-Hamiltonian}
\newacronym{dmd}{DMD}{dynamic mode decomposition}
\newacronym{opinf}{OpInf}{operator inference}
\newacronym{mor}{MOR}{model order reduction}
\newacronym{ode}{ODE}{ordinary differential equation}
\newacronym{dae}{DAE}{differential-algebraic equation}
\newacronym{pde}{PDE}{partial differential equation}
\newacronym{pinn}{PINN}{physics-informed neural network}
\newacronym{sindy}{SINDy}{sparse identification of nonlinear dynamics}
\newacronym{ae}{AE}{autoencoder}
\newacronym{kyp}{KYP}{Kalman-Yakubovich-Popov}
\newacronym{pr}{PR}{positive-real}
\newacronym{phdmd}{pHDMD}{port-Hamiltonian dynamic mode decomposition}
\newacronym{fgm}{FGM}{fast-gradient method}
\newacronym{pod}{POD}{proper orthogonal decomposition}
\newacronym{lti}{LTI}{linear time-invariant}
\newacronym{fom}{FOM}{full-order model}
\newacronym{rom}{ROM}{reduced-order model}
\newacronym{lhs}{LHS}{left-hand side}
\newacronym{rhs}{RHS}{right-hand side}
\newacronym{fe}{FE}{finite element}
\newacronym{fem}{FEM}{finite element method}
\newacronym{msd}{MSD}{mass-spring-damper}
\newacronym{mimo}{MIMO}{multiple-input multiple-output}
\newacronym{imr}{IMR}{implicit midpoint rule}

%% file: figures/ph_figure_config.tex
\usepackage{tikz, tikzscale}
\usepackage{circuitikz}
\usepackage{pgfplots}
\usepgfplotslibrary{colorbrewer}
\pgfplotsset{compat=newest}
\usetikzlibrary{matrix, calc, positioning, fit, patterns, decorations.pathmorphing, decorations.markings, positioning, angles, shapes.geometric, math, arrows.meta, fadings, through, quotes,pgfplots.statistics}
\usetikzlibrary{intersections, backgrounds}

\usepackage{pgfplotstable}

\pgfplotsset{
  every axis plot/.append style={line width=1pt},
  every axis/.append style={
    ymajorgrids,
    xmajorgrids,
    grid style={dashed, lightgray,semithick},
   	axis line style = semithick,
   	every tick/.style=semithick,
    enlarge x limits=false,
    enlarge y limits=true,
  },
  every axis legend/.append style={
        fill opacity=0.7,
        draw opacity=1,
        text opacity=1,
        legend cell align = left,
        text=black,
        /tikz/every even column/.append style={column sep=10pt},
    },
  myline/.style={
    thick,
  },
  mybar/.style n args={1}{
    fill=#1!30!white,
    draw=#1,
  },
  mylegend/.style={
        fill opacity=0.7,
        draw opacity=1,
        text opacity=1,
        draw=white!80!black,
    },
  general_mark_options/.style={
        mark size = 1.8, 
        mark repeat = 50,
        mark options={solid},
  },
  identified_line/.style={        
        dashed,
        general_mark_options,
        mark = diamond*,
        mark phase = 25,
        },
  intrusive_line/.style={        
        dashdotted,
        general_mark_options,
        mark = square*,
        mark phase = 45,
  },
  error_line/.style={
        densely dotted,
  },
  input1_line/.style={
        myline,
        color=plotColor1,
  },
  input2_line/.style={
        myline,
        color=plotColor2,
        dashed,
  },
  diff_id_intrusive_line/.style={
        identified_line,
        mark repeat =1,
        mark phase = 0,
  },
  diff_id_intrusive_with_bias_line/.style={
        intrusive_line,
        mark repeat =1,
        mark phase = 0,
  },
  cond_rhs_data_reduced_line/.style={
        error_line,
        general_mark_options,
        mark = diamond*,
        mark repeat =1,
        mark phase = 0,
  },
  reference_line/.style={  
        myline, 
        reference_color,     
        solid,
        general_mark_options,
        mark = *, 
        mark phase=10,    
  },
galerkin_intrusive_bar/.style={
        pattern=north east lines,
        pattern color=opinf_galerkin_color,
        draw=opinf_galerkin_color,
  },
  petrov_galerkin_intrusive_bar/.style={
         pattern=north east lines,
        pattern color=opinf_petrov_galerkin_color,
        draw=opinf_petrov_galerkin_color,
  },
  opinf_galerkin_line/.style={
        myline,
        color=opinf_galerkin_color,
  },
  opinf_galerkin_bar/.style={
        mybar={opinf_galerkin_color},
  },
opinf_petrov_galerkin_line/.style={
        myline,
        color=opinf_petrov_galerkin_color,
},
opinf_petrov_galerkin_bar/.style={
        mybar={opinf_petrov_galerkin_color},
},
opinf_petrov_galerkin_hamiltonian_line/.style={
        myline,
        color=opinf_petrov_galerkin_hamiltonian_color,
},
opinf_petrov_galerkin_hamiltonian_bar/.style={
        mybar={opinf_petrov_galerkin_hamiltonian_color},
},
opinf_petrov_galerkin_kyp_line/.style={
        myline,
        color=opinf_petrov_galerkin_kyp_color,
},
opinf_petrov_galerkin_kyp_bar/.style={
        mybar={opinf_petrov_galerkin_kyp_color},
},
cvxph_galerkin_line/.style={
        myline,
        color=cvxph_galerkin_color,
},
cvxph_galerkin_bar/.style={
        mybar={cvxph_galerkin_color},
},
cvxph_petrov_galerkin_line/.style={
        myline,
        color=cvxph_petrov_galerkin_color,
},
cvxph_petrov_galerkin_bar/.style={
        mybar={cvxph_petrov_galerkin_color},
},
galerkin_line/.style={
        myline,
        general_mark_options,
        mark = *, 
        mark phase=0,
        mark repeat =2,
},
petrov_galerkin_line/.style={
        myline, 
        general_mark_options,
        mark = diamond*, 
        mark phase=2,
        mark repeat =2,
        },
rel_error_ABCD_line/.style={
        myline,
        color=rel_error_ABCD_color,
},
rel_error_ABCD_with_residual_line/.style={
        myline,
        color=rel_error_ABCD_with_residual_color,
},
condition_number_line/.style={
        myline,
        color=condition_number_color,
},
rel_error_A_B_line/.style={
        myline,
        color=rel_error_A_B_color,
},
rel_error_C_D_line/.style={
        myline,
        color=rel_error_C_D_color,
},
residual_term_states_line/.style={
        myline,
        color=residual_term_states_color,
        dashed,
},
residual_term_outputs_line/.style={
        myline,
        color=residual_term_outputs_color,
        dashed,
},
residual_term_states_without_pinv_line/.style={
        myline,
        color=residual_term_states_without_pinv_color,
},
residual_term_outputs_without_pinv_line/.style={
        myline,
        color=residual_term_outputs_without_pinv_color,
},}

\tikzset{mark options/.append style = {mark repeat = 60}, % repeat mark on every n points
}

\pgfplotsset{
        gen_mark_opt/.style={
                mark options={solid},
                mark size = 1.8,
                mark repeat = 1,  
        },
        err_ops_line/.style={
                myline,
                color=Set1-A,
                gen_mark_opt,
                mark=o,
        },
        norm_corr_mark/.style={
                gen_mark_opt,
                mark=x,
                only marks,
        },
        err_bound_line/.style={
                myline,
                color=Set1-C,
                gen_mark_opt,
                mark=triangle*,
        },
        err_est_line/.style={
                myline,
                color=Set1-D,
                mark options={rotate=180}, % use mark options before further options, e.g. mark repeat, mark phase are defined
                gen_mark_opt,
                mark=triangle*,
        },
        sing_val_line/.style={
                myline,
                color=Set1-E,
                gen_mark_opt,
                mark=diamond*,
        },
        gpod_line/.style={
                myline,
                color=Set1-A,
                gen_mark_opt,
                mark=*,
        },
        pgopinfham_line/.style={
                gpod_line,
        },
        gpod_stab_line/.style={
                gpod_line,
                dashed,
                mark=o,
        },
        pgopinfham_stab_line/.style={
                gpod_stab_line,
        },
        pgpod_line/.style={
                myline,
                color=Set1-C,
                gen_mark_opt,
                mark=square*,
        },
        cvxph_line/.style={
                pgpod_line,
        },
        gopinf_line/.style={
                myline,
                color=Set1-D,
                gen_mark_opt,
                mark=triangle*,
        },
        gopinf_stab_line/.style={
                gopinf_line,
                dashed,
                mark=triangle,
        },
        pgopinf_line/.style={
                myline,
                color=Set1-E,
                gen_mark_opt,
                mark=diamond*,
        },
        pgopinf_stab_line/.style={
                pgopinf_line,
                dashed,
                mark=diamond,
        },
}

\definecolor{plotColor1}{HTML}{59a89c} % #59a89c
\definecolor{plotColor2}{HTML}{e02b35} % #e02b35
\definecolor{plotColor3}{HTML}{082a54}  % #082a54
\definecolor{plotColor4}{HTML}{a559aa} % #a559aa
\definecolor{plotColor5}{HTML}{ff7f0e}  % #ff7f0e
\definecolor{plotColor6}{HTML}{2ca02c}  % #2ca02c
\definecolor{plotColor7}{HTML}{1f77b4}  % #1f77b4
\definecolor{plotColor8}{HTML}{f0c571}  % #f0c571
\definecolor{plotColor9}{HTML}{cecece}  % #cecece

\definecolor{gray1}{gray}{0.15}     % #262626
\definecolor{gray2}{gray}{0.3}      % #4D4D4D
\definecolor{gray3}{gray}{0.45}     % #737373
\definecolor{gray4}{gray}{0.6}      % #999999
\definecolor{gray5}{gray}{0.75}     % #BFBFBF
\definecolor{gray6}{gray}{0.9}      % #E6E6E6

\colorlet{reference_color}{black}
\colorlet{opinf_galerkin_color}{plotColor1}
\colorlet{opinf_petrov_galerkin_color}{plotColor2}
\colorlet{opinf_petrov_galerkin_hamiltonian_color}{plotColor3}
\colorlet{opinf_petrov_galerkin_kyp_color}{plotColor4}
\colorlet{cvxph_galerkin_color}{plotColor5}
\colorlet{cvxph_petrov_galerkin_color}{plotColor6}
\colorlet{zoom_box_color}{gray2}
\colorlet{rel_error_ABCD_color}{plotColor1}
\colorlet{rel_error_ABCD_with_residual_color}{plotColor2}
\colorlet{condition_number_color}{plotColor3}
\colorlet{rel_error_A_B_color}{plotColor4}
\colorlet{rel_error_C_D_color}{plotColor5}
\colorlet{residual_term_states_color}{plotColor6}
\colorlet{residual_term_outputs_color}{plotColor7}
\colorlet{residual_term_states_without_pinv_color}{plotColor8}
\colorlet{residual_term_outputs_without_pinv_color}{plotColor9}

%% file: plots/cdplayer_opinf_metrics.tikz
\newcommand{\cdplayerfilename}{cd_player_opinf_metrics_over_r.csv}

\begin{tikzpicture}
\begin{axis}[
    width=.94\linewidth,
    height=0.6\linewidth,
    ymode=log,
    xlabel={Reduced order $\dimReducedStates$},
    ylabel={Rel. error},
    ymin=1e-6,
    ymax=1e3,
    ytick={1e-6,1e-4,1e-2,1e0,1e2},
    legend to name = leg:cdplayer,
    legend columns=5,
    % legend style={
    % text=black, column sep=2pt,
    % /tikz/every even column/.append style={column sep=10pt}
    % }
]

\addplot[err_ops_line] table [x=r, y=eps_ref, col sep=comma] {\cdplayerfilename};
\addlegendentry{$\errOps$}

\addplot[norm_corr_mark] table [x=r, y=eps_corr, col sep=comma] {\cdplayerfilename};
\addlegendentry{$\normCorr$}

\addplot[err_bound_line] table [x=r, y=eps_bound, col sep=comma] {\cdplayerfilename};
\addlegendentry{$\errBound$}

\addplot[err_est_line] table [x=r, y=eps_est, col sep=comma] {\cdplayerfilename};
\addlegendentry{$\errEst$}

\addplot[sing_val_line] table [x=r, y=sigma, col sep=comma] {\cdplayerfilename};
\addlegendentry{$\sigma$}

\end{axis}
\end{tikzpicture}

%% file: plots/earth_atmosphere_opinf_metrics.tikz
\newcommand{\earthatmospherefilename}{earth_atmosphere_opinf_metrics_over_r.csv}

\begin{tikzpicture}
\begin{axis}[
    width=.94\linewidth,
    height=0.6\linewidth,
    ymode=log,
    xlabel={Reduced order $\dimReducedStates$},
    xmin=2,
    xmax=80,
    ylabel={Rel. error},
    ymin=1e-6,
    ymax=1e3,
    ytick={1e-6,1e-4,1e-2,1e0,1e2},
    legend to name = leg:earthatmosphere,
    legend columns=5,
    % legend style={
    % text=black, column sep=2pt,
    % /tikz/every even column/.append style={column sep=10pt}
    % }
]

\addplot[err_ops_line] table [x=r, y=eps_ref, col sep=comma] {\earthatmospherefilename};
\addlegendentry{$\errOps$}

\addplot[norm_corr_mark] table [x=r, y=eps_corr, col sep=comma] {\earthatmospherefilename};
\addlegendentry{$\normCorr$}

\addplot[err_bound_line] table [x=r, y=eps_bound, col sep=comma] {\earthatmospherefilename};
\addlegendentry{$\errBound$}

\addplot[err_est_line] table [x=r, y=eps_est, col sep=comma] {\earthatmospherefilename};
\addlegendentry{$\errEst$}

\addplot[sing_val_line] table [x=r, y=sigma, col sep=comma] {\earthatmospherefilename};
\addlegendentry{$\sigma$}

\end{axis}
\end{tikzpicture}

%% file: plots/msd_opinf_metrics.tikz
\newcommand{\msdfilename}{msd_mimo_50mass_opinf_metrics_over_r.csv}

\begin{tikzpicture}
\begin{axis}[
    width=.94\linewidth,
    height=0.6\linewidth,
    ymode=log,
    xlabel={Reduced order $\dimReducedStates$},
    ylabel={Rel. error},
    ymin=1e-6,
    ymax=1e3,
    ytick={1e-6,1e-4,1e-2,1e0,1e2},
    legend to name = leg:msd,
    legend columns=5,
    % legend style={
    % text=black, column sep=2pt,
    % /tikz/every even column/.append style={column sep=10pt}
    % }
]

\addplot[err_ops_line] table [x=r, y=eps_ref, col sep=comma] {\msdfilename};
\addlegendentry{$\errOps$}

\addplot[norm_corr_mark] table [x=r, y=eps_corr, col sep=comma] {\msdfilename};
\addlegendentry{$\normCorr$}

\addplot[err_bound_line] table [x=r, y=eps_bound, col sep=comma] {\msdfilename};
\addlegendentry{$\errBound$}

\addplot[err_est_line] table [x=r, y=eps_est, col sep=comma] {\msdfilename};
\addlegendentry{$\errEst$}

\addplot[sing_val_line] table [x=r, y=sigma, col sep=comma] {\msdfilename};
\addlegendentry{$\sigma$}

\end{axis}
\end{tikzpicture}

%% file: plots/poro_opinf_metrics.tikz
\newcommand{\porofilename}{poro_opinf_metrics_over_r.csv}

\begin{tikzpicture}
\begin{axis}[
    width=.94\linewidth,
    height=0.6\linewidth,
    ymode=log,
    xlabel={Reduced order $\dimReducedStates$},
    ylabel={Rel. error},
    ymin=1e-6,
    ymax=1e3,
    ytick={1e-6,1e-4,1e-2,1e0,1e2},
    legend to name = leg:poro,
    legend columns=5,
    % legend style={
    % text=black, column sep=2pt,
    % /tikz/every even column/.append style={column sep=10pt}
    % }
]

\addplot[err_ops_line] table [x=r, y=eps_ref, col sep=comma] {\porofilename};
\addlegendentry{$\errOps$}

\addplot[norm_corr_mark] table [x=r, y=eps_corr, col sep=comma] {\porofilename};
\addlegendentry{$\normCorr$}

\addplot[err_bound_line] table [x=r, y=eps_bound, col sep=comma] {\porofilename};
\addlegendentry{$\errBound$}

\addplot[err_est_line] table [x=r, y=eps_est, col sep=comma] {\porofilename};
\addlegendentry{$\errEst$}

\addplot[sing_val_line] table [x=r, y=sigma, col sep=comma] {\porofilename};
\addlegendentry{$\sigma$}

\end{axis}
\end{tikzpicture}

%% file: plots/msd_hinf_exp2.tikz
\newcommand{\msdHinfData}{msd_mimo_50mass_hinf_error_petrov_vs_galerkin.csv}

\begin{tikzpicture}
\begin{axis}[
    width=0.93\linewidth,
    height=0.6\linewidth,
    ymode=log,
    xlabel={Reduced order $\dimReducedStates$},
    unbounded coords=jump,
    ylabel={Rel. $\Hinf$-error},
    % ymin=1e-2,
    % ymax=1e2,
    legend to name = leg:msd:hinf,
    legend columns=5,
    legend style={
    text=black, column sep=2pt,
    /tikz/every even column/.append style={column sep=10pt}
    }
]

% gpod-stab
\addplot[gpod_stab_line, forget plot] table [x=r, y=pod_opinf_g | intrusive | True, col sep=comma] {\msdHinfData};

% gpod
\addplot[gpod_line] table [x=r, y=pod_opinf_g | intrusive | False, col sep=comma] {\msdHinfData};
\addlegendentry{\gpod}

% gopinf-stab
\addplot[gopinf_stab_line, forget plot] table [x=r, y=pod_opinf_g | identified | True, col sep=comma] {\msdHinfData};

% gopinf
\addplot[gopinf_line] table [x=r, y=pod_opinf_g | identified | False, col sep=comma] {\msdHinfData};
\addlegendentry{\gopinf}

% pgpod
\addplot[pgpod_line] table [x=r, y=pod_opinf_pg | intrusive | False, col sep=comma] {\msdHinfData};
\addlegendentry{\pgpod}

% pgopinf
\addplot[pgopinf_line] table [x=r, y=pod_opinf_pg | identified | False, col sep=comma] {\msdHinfData};
\addlegendentry{\pgopinf}

% pgopinf-stab
\addplot[pgopinf_stab_line, forget plot] table [x=r, y=pod_opinf_pg | identified | True, col sep=comma] {\msdHinfData};
\addlegendentry{\pgopinf}

\end{axis}
\end{tikzpicture}

%% file: plots/poro_hinf_exp2.tikz
\newcommand{\poroHinfData}{poro_hinf_error_petrov_vs_galerkin.csv}

\begin{tikzpicture}
\begin{axis}[
    width=0.93\linewidth,
    height=0.6\linewidth,
    ymode=log,
    xlabel={Reduced order $\dimReducedStates$},
    unbounded coords=jump,
    ylabel={Rel. $\Hinf$-error},
    % ymin=1e-2,
    % ymax=1e2,
    legend to name = leg:poro:hinf,
    legend columns=5,
    legend style={
    text=black, column sep=2pt,
    /tikz/every even column/.append style={column sep=10pt}
    }
]

% gpod-stab
\addplot[gpod_stab_line, forget plot] table [x=r, y=pod_opinf_g | intrusive | True, col sep=comma] {\poroHinfData};

% gpod
\addplot[gpod_line] table [x=r, y=pod_opinf_g | intrusive | False, col sep=comma] {\poroHinfData};
\addlegendentry{\gpod}

% gopinf-stab
\addplot[gopinf_stab_line, forget plot] table [x=r, y=pod_opinf_g | identified | True, col sep=comma] {\poroHinfData};

% gopinf
\addplot[gopinf_line] table [x=r, y=pod_opinf_g | identified | False, col sep=comma] {\poroHinfData};
\addlegendentry{\gopinf}

% pgpod
\addplot[pgpod_line] table [x=r, y=pod_opinf_pg | intrusive | False, col sep=comma] {\poroHinfData};
\addlegendentry{\pgpod}

% pgopinf
\addplot[pgopinf_line] table [x=r, y=pod_opinf_pg | identified | False, col sep=comma] {\poroHinfData};
\addlegendentry{\pgopinf}

% pgopinf-stab
\addplot[pgopinf_stab_line, forget plot] table [x=r, y=pod_opinf_pg | identified | True, col sep=comma] {\poroHinfData};
\addlegendentry{\pgopinf}

\end{axis}
\end{tikzpicture}

%% file: plots/msd_hinf_exp3.tikz
\newcommand{\msdhinffilename}{msd_mimo_50mass_hinf_error_g_pg_hamcvx_ph.csv}

\def\yMin{-0.1}
\def\yMax{0.1}
\def\xMin{0}
\def\xMax{50}

\def\widthFactor{0.45}

\newcommand{\addstabilizedfalse}[1]{#1 | False}
\newcommand{\addstabilizedtrue}[1]{#1 | True}

\begin{tikzpicture}
    % Simulation 0
    \begin{axis}
        [   width=0.7\textwidth,
            height=0.6211*0.7*\textwidth,
            xlabel=Reduced order $\dimReducedStates$,
            ylabel={Rel. $\Hinf$-error},
            ymode = log,
            unbounded coords=jump,
            % title=test trajectory 1,
            at = {(0,0)},
            legend pos = north east,
            % ymin = \yMin,
            % ymax = \yMax,
            % restrict y to domain=\yMin:\yMax,
            % xmin = \xMin,
            % xmax = \xMax,
            % restrict x to domain=\xMin:\xMax,
            % scale only axis,
            name=msd_output_exp3,
            legend columns=4,
            legend to name=leg:msd:exp3:hinf,
        ]
        % \addplot [reference_line, mark=none,
        %     ] table [col sep=comma, x=time, y=pod_opinf_g | identified] {\msdhinffilename};
        % \addlegendentry{reference}
        \addplot [gopinf_line] table [col sep=comma, x=r, y=\addstabilizedfalse{pod_opinf_g | identified}] {\msdhinffilename};  
        \addlegendentry{\gopinf} 
        \addplot [gopinf_stab_line, forget plot, dashed] table [col sep=comma, x=r, y=\addstabilizedtrue{pod_opinf_g | identified}] {\msdhinffilename};  
        % \addlegendentry{\gopinf}
        \addplot [pgopinf_line] table [col sep=comma, x=r, y=\addstabilizedfalse{pod_opinf_pg_Qknown | identified}] {\msdhinffilename};  
        \addlegendentry{\pgopinf}   
        \addplot [pgopinfham_line] table [col sep=comma, x=r, y=\addstabilizedfalse{pod_opinf_pg_Qhamcvx | identified}] {\msdhinffilename};  
        \addlegendentry{\pgopinfham}   
        \addplot [pgopinfham_stab_line, forget plot] table [col sep=comma, x=r, y=\addstabilizedtrue{pod_opinf_pg_Qhamcvx | identified}] {\msdhinffilename};  
        % \addlegendentry{\pgopinfham} 
        \addplot [cvxph_line] table [col sep=comma, x=r, y=\addstabilizedfalse{pod_convex_ph | identified}] {\msdhinffilename};  
        \addlegendentry{\cvxph} 
    \end{axis}
\end{tikzpicture}